\documentclass[a4paper,11pt]{amsart}
\usepackage[toc,page]{appendix} 
\usepackage[french,english]{babel} 
\usepackage[utf8]{inputenc} 
\usepackage[T1]{fontenc}     
\usepackage{graphicx}
\usepackage{color}
\usepackage{textcomp}
\usepackage{amsmath}
\usepackage{enumerate}
\usepackage{mathrsfs}
\usepackage{multirow}
\usepackage{amsthm}
\usepackage{amsfonts}
\usepackage[all]{xy}
\usepackage{yhmath}
\usepackage[hidelinks]{hyperref}
\usepackage{amssymb}
\usepackage{geometry}
\usepackage{import}

\newcommand{\rr}{\mathbb{R}}

\newcommand{\cc}{\mathcal{C}}
\newcommand{\hh}{\mathcal{H}}

\newcommand{\aaa}{\mathcal{A}}

\newcommand{\ff}{\mathcal{F}}

\newcommand{\dd}{\partial}
\newcommand{\ep}{\varepsilon}

\newtheorem{theo}{Theorem}
\newtheorem{prop}{Proposition}[section]

\newtheorem{lem}[prop]{Lemma}
\newtheorem{cor}[prop]{Corollary}
\theoremstyle{definition} 
\newtheorem{defi}[prop]{Definition}
\newtheorem{hyp1}[prop]{Hypothesis}

\theoremstyle{remark}
\newtheorem{rem}[prop]{Remark}

\theoremstyle{remark}
\newtheorem*{exe}{Example}
\theoremstyle{definition} 

\theoremstyle{remark}

\definecolor{Green}{RGB}{0,128,0}
\definecolor{Green2}{RGB}{0,128,0}
\definecolor{Orange}{RGB}{255,165,0}
\title[Non-existence of characteristics for viscosity solutions]{Non-existence of global characteristics for viscosity solutions}
\author[V. Roos]{Valentine Roos\\ \\ \today}\thanks{The research leading to these results has received funding from the European Research Council under
	the European Union’s Seventh Framework Programme (FP/2007-2013) / ERC Grant Agreement 307062 and from the French National Research Agency via ANR-12-BLAN-WKBHJ}
\begin{document}
	
	\maketitle

\begin{abstract} Two different types of generalized solutions, namely viscosity and variational solutions, were introduced to solve the first-order evolutionary Hamilton--Jacobi equation. They coincide if the Hamiltonian is convex in the momentum variable. In this paper we prove that there exists no other class of integrable Hamiltonians sharing this property. To do so, we build for any non-convex non-concave integrable Hamiltonian a smooth initial condition such that the graph of the viscosity solution is not contained in the wavefront associated with the Cauchy problem. The construction is based on a new example for a saddle Hamiltonian and a precise analysis of the one-dimensional case, coupled with reduction and approximation arguments.
\end{abstract}

\section{Introduction}

Let $H : \rr \times T^\star \rr^d \to \rr $ be a $\cc^2$ Hamiltonian. 
We study the Cauchy problem associated with the evolutionary Hamilton--Jacobi equation
\begin{equation}\tag{HJ}\label{HJ}
\dd_t u (t,q) + H(t,q,\dd_q u (t,q))=0
\end{equation}
where $u : \rr\times \rr^d \to \rr$ is the unknown function, with a Lipschitz initial datum $u(0,\cdot)=u_0$.

The method of characteristics shows that a classical solution of this equation is given by characteristics (see \S \ref{introclass}). If the projections of characteristics associated with $u_0$ cross, the method gives rise to a multivalued solution, with a multigraph called \emph{wavefront} and denoted by $\ff_{u_0}$ (see \eqref{defwf}). This implies in particular that for some $u_0$ and $H$, even smooth, the evolutionary Hamilton--Jacobi equation does not admit classical solutions in large time. 

A first type of generalized solution, called \emph{viscosity solution} (see \S \ref{introvisc}), was introduced by Lions, Crandall and Evans in the early 80s for Hamilton-Jacobi equations. It presents multiple assets: it is well defined, unique and stable in a large range of assumptions on the Hamiltonian and the initial condition. It has a local definition, which allows to avoid the delicate question of how to choose a solution amongst the multivalued solution and its associated characteristics. This local definition can be extended effortless to larger classes of elliptic PDEs, which is an other major asset of viscosity solutions.
Also, the operator giving the viscosity solution satisfies a convenient semigroup property. 

When the Hamiltonian is convex in the fiber (more precisely when it is Tonelli), this viscosity operator is given by the Lax--Oleinik semigroup, which by definition gives a section of the wavefront. The main result of this article addresses the converse interrogation, in the case of integrable (\emph{i.e.} depending only on the fiber variable) Hamiltonians.
\begin{theo}\label{joukrec}
	If  $p\mapsto H(p)$ is a neither convex nor concave integrable Hamiltonian with bounded second derivative, there exists a smooth Lipschitz initial condition $u_0$ such that the graph of the viscosity solution associated with $u_0$ is not included in the wavefront $\ff_{u_0}$.
\end{theo}

The term of \emph{variational solution} (see \S \ref{introvar}) does not appear in this statement but the idea of this other generalized solution is pregnant in the whole article: roughly speaking, they can be defined as continuous functions whose graph is included in the wavefront. The notion was introduced in the early 90s by Sikorav and Chaperon, who find a way to choose a continuous section of the wavefront by selecting the minmax value of the generating family for the Lagrangian geometrical solution. Joukovskaia showed in \cite{jou} that their construction coincide with the Lax--Oleinik semigroup in the fiberwise convex case. The study of the variational operator given by this Chaperon--Sikorav method gives local estimates on the variational solutions. These estimates can be used regardless of the construction of the variational solution thanks to Proposition \ref{varmin}, which gives an elementary characterization of the variational solution for semiconcave initial data. This fact makes the whole article accessible to a reader with no specific background on symplectic geometry.

To show Theorem \ref{joukrec}, we reduce the problem to the study of two key situations in dimension $1$ and $2$, see Propositions  \ref{dim1ex} and Proposition \ref{contrex}. The example for the dimension $1$ was already well studied: it appears in \cite{chenciner}, see also \cite{izukos}. The creation of the example for the saddle Hamiltonian in dimension $2$ is the main contribution of this article. A special care was then provided to state the reduction and approximation arguments finishing the demonstration. 

Recent breakthroughs have been made in the study of the singularities of the viscosity solution of \eqref{HJ} for convex Hamiltonians, see \cite{canmazsin}, \cite{canchefat}, or \cite{canchesurvey} for a survey. A natural question following from Theorem \ref{joukrec}  is to compare these singularities for viscosity and variational solutions when the Hamiltonian is not convex anymore. On the close topic of multi-time Hamilton-Jacobi equations, let us also highlight a recent discussion about the non-existence of viscosity solutions when convexity assumptions are dropped, see \cite{davzav}. This gives another point of comparison with variational solutions, that are well-defined for this framework, see \cite{carvit}.

Since Proposition \ref{varmin} holds for non integrable Hamiltonians, we present the different objects in the non integrable framework. We will underline how they simplify in the integrable case. In that purpose, we introduce a second Hypothesis on $H$, automatically satisfied by integrable Hamiltonians with bounded second derivative, that provides the existence of both viscosity and variational solutions in the non integrable case.		
\begin{hyp1}\label{esti} There is a $C>0$ such that for each $(t,q,p)$ in $\rr\times \rr^d \times \rr^d$,\[
	\|\dd^2_{(q,p)} H(t,q,p)\| < C,\; \|\dd_{(q,p)} H(t,q,p)\| < C(1+\|p\|),
	\]where $\dd_{(q,p)} H$ and $\dd^2_{(q,p)} H$ denote the first and second order spatial derivatives of $H$. \end{hyp1}

\subsection{Classical solutions: the method of characteristics}\label{introclass}
In this section we only assume that $d^2H$ is bounded by $C$. The \emph{Hamiltonian system}
\begin{equation}\label{HS}\tag{HS}\left\{\begin{array}{l}
\dot{q}(t)= \dd_p H(t,q(t),p(t)),\\
\dot{p}(t)= - \dd_q H(t,q(t),p(t))
\end{array}\right.
\end{equation}
hence admits a complete \emph{Hamiltonian flow} $\phi^t_s$, meaning that $t\mapsto \phi^t_s(q,p)$ is the unique solution of \eqref{HS} with initial conditions $\left(q(s),p(s)\right)=(q,p)$. We denote by $\left(Q^t_s,P^t_s\right)$ the coordinates of $\phi^t_s$. 
We call a function $t \mapsto (q(t),p(t))$ solving the Hamiltonian system \eqref{HS} a \emph{Hamiltonian trajectory}.
The \emph{Hamiltonian action} of a $\cc^1$ path $\gamma(t)=\left(q(t),p(t)\right)\in T^\star \rr^d $ is denoted by \[
\aaa^t_s (\gamma)=\int^t_s p(\tau)\cdot \dot{q}(\tau)-H(\tau,q(\tau),p(\tau))d\tau.
\]

Note that in the case of an integrable Hamiltonian (that depends only on $p$), the flow is given by $\phi^t_s(q,p)=(q+(t-s)\nabla H(p),p)$ and the action of an Hamiltonian path is reduced to $\aaa^t_s(\gamma)=(t-s)(p\cdot \nabla H(p)-H(p))$. 

The method of characteristics states that if $u_0$ is a $\cc^2$ function with second derivative bounded by $B>0$, there exists $T$ depending only on $C$ and $B$ (for example $T=1/BC$ for an integrable Hamiltonian) such that the Cauchy problem \eqref{HJ} with initial condition $u_0$ has a unique $\cc^2$ solution on $[0,T]\times \rr^d \to \rr$. Furthermore, if $u$ is a $\cc^2$ solution on $[0,T]\times \rr^d$, for all $(t,q)$ in $[0,T]\times\rr^d$, there exists a unique $q_0$ in $\rr^d$ such that $Q^t_0(q_0,du_0(q_0))=q$ and if $\gamma$ denotes the Hamiltonian trajectory issued from $(q_0,du_0(q_0))$, the $\cc^2$ solution is given by the Hamiltonian action as follows:
	\[u(t,q)=u_0(q_0)+ \aaa^t_0(\gamma),\]
and its derivative satisfies $\dd_q u(t,q)=P^t_0(q_0,du_0(q))$ at the point $q=Q^t_0(q_0,du_0(q_0))$. As a consequence, if the image $\phi^t_0\left(gr(du_0)\right)$  of the graph of $du_0$ by the Hamiltonian flow is not a graph for some $t$, there is no classical solution on $[0,t]\times \rr^d$, whence the necessity to introduce generalized solutions.

\subsection{Viscosity solutions}\label{introvisc}
The viscosity solutions were introduced in the framework of Hamilton--Jacobi equations by Lions, Evans and Crandall in the early $80$'s, see \cite{cr&lions83}. We will use the following definition.
\begin{defi}\label{defivisc} A continuous function $u$ is a \emph{subsolution} of \eqref{HJ} on the set $(0,T)\times \rr^d$ if for each $\cc^1$ function $\phi:(0,T)\times \rr^d\to \rr$ such that $u-\phi$ admits a (strict) local maximum at a point $(t,q)\in(0,T)\times \rr^d$, \[\dd_t \phi(t,q)+H(t,q,\dd_q\phi(t,q))\leq 0.\]  			A continuous function $u$ is a \emph{supersolution} of \eqref{HJ} on the set $(0,T)\times \rr^d\to \rr$ if for each $\cc^1$ function $\phi:(0,T)\times \rr^d$ such that $u-\phi$ admits a (strict) local minimum at a point $(t,q)\in(0,T)\times \rr^d$, \[\dd_t \phi(t,q)+H(t,q,\dd_q\phi(t,q))\geq 0.\]  		
	A viscosity solution is both a sub- and supersolution of \eqref{HJ}.
\end{defi} 	

The set of assumptions of this paper is well adapted to the theory of viscosity solutions developed by Crandall, Lions and Ishii (see \cite{guide}), from which one can deduce the following well-posedness property.
\begin{prop}\label{visc}
	If $H$ satisfies Hypothesis \ref{esti}, the Cauchy problem associated with the \eqref{HJ} equation and a Lipschitz initial condition admits a unique Lipschitz solution. This defines a viscosity operator $\left(V^t_s\right)_{s\leq t}$ on the set of Lipschitz functions $\cc^{0,1} (\rr^d)$ which is monotonous: \[
	V^t_s u \leq V^t_s v \;\textrm{ if } \;u \leq v.
	\]
	Furthermore, if $u$ and $v$ are Lipschitz with bounded difference, 
	\[ \|  V^t_s u - V^t_s v   \|_\infty \leq \|u-v\|_\infty \,\; \textrm{ for all } \,\;s\leq t.\]
\end{prop}

In dimension $1$, the theory of viscosity solutions of the \eqref{HJ} equation is the counterpart of the theory of entropy solutions for conservation laws: if $p(t,q)=\dd_q u(t,q)$ and $u$ satisfies \eqref{HJ},
\[\dd_t p(t,q) + \dd_q(H(t,q,p(t,q)))=0.\]
The following entropy condition, first proposed by O. Oleinik in \cite{oleinik} for conservation laws, gives a geometric criterion to decide if a function solves the \eqref{HJ} equation in the viscosity sense at a point of shock. It is proved for example in \cite{kos} (Theorem $2.2$) in the modern viscosity terms, as a direct application of Theorem $1.3$ in \cite{CEL}. We give the statement for $H$ integrable, \emph{i.e.} which depends only on $p$.
\begin{defi}[Oleinik's entropy condition]\label{oleinik} Let $H:\rr \to \rr$ be a $\cc^2$ Hamiltonian. If $(p_1,p_2)\in \rr^2$, we say that \emph{Oleinik's entropy condition}
	is (strictly) satisfied between $p_1$ and $p_2$ if 
	\[H(\mu p_1+(1-\mu)p_2) \overset{(<)}{\leq} \mu H(p_1)+(1-\mu)H(p_2) \,\;\, \forall \mu \in (0,1),\]
	\emph{i.e.}  if and only if the graph of $H$ lies (strictly) under the cord joining $(p_1,H(p_1))$ and $(p_2,H(p_2))$.
	
	We say that the \emph{Lax condition} is (strictly) satisfied if 
	\[ H'(p_1)(p_2-p_1) \overset{(<)}{\leq} {H(p_2)-H(p_1)} \overset{(<)}{\leq} H'(p_2)(p_2-p_1),\]
	which is implied by the entropy condition.
\end{defi}
\begin{prop}\label{RH} Let $u=\min(f_1,f_2)$ on an open neighbourhood $U$ of $(t,q)$ in $\rr_+ \times \rr$, with $f_1$ and $f_2$ $\cc^1$ solutions on $U$ of the Hamilton--Jacobi equation \eqref{HJ}. Let $p_1$ and $p_2$ denote respectively $\dd_q f_1(t,q)$ and $\dd_q f_2(t,q)$. If $f_1(t,q)=f_2(t,q)$, then $u$ is a viscosity solution at $(t,q)$ if and only if the {entropy condition} is satisfied between $p_1$ and $p_2$.
\end{prop}

Oleinik's entropy condition is also valid in higher dimensions (for shock along a smooth hypersurface), see Theorem $3.1$ in \cite{izukos}, and can be generalized when $u$ is the minimum of more than two functions, see \cite{bernard2}.
\subsection{Variational solutions}\label{introvar}
If $u_0$ is a $\cc^1$ initial condition, the \emph{wavefront} associated with the Cauchy problem for $u_0$ is denoted by $\mathcal{F}_{u_0}$ and defined by \begin{displaymath} \tag{F} \label{defwf}
\mathcal{F}_{u_0}=\left\{\left(t,q,u_0(q_0)+ \aaa^t_0(\gamma)\right) \left|\begin{array}{c}
t\geq 0, q\in \rr^d,\\
p_0 = d u_0(q_0),\\
Q^t_0(q_0,p_0)=q.
\end{array} \right\}\right.
\end{displaymath}
In terms of wavefront, the method of characteristics explained in \S \ref{introclass} states that if $u$ is a $\cc^2$ solution on $[0,T]\times \rr^d$, the restrictions on $[0,T]\times \rr^d$ of the graph of $u$ and of the wavefront coincide.

If $u_0$ is $\cc^1$, we will call a \emph{variational solution} of the Cauchy problem associated with $u_0$ a continuous function whose graph is included in the wavefront $\ff_{u_0}$, \emph{i.e.} a continuous function $g:[0,T]\times \rr^d$ such that for all $(t,q)$ in $[0,\infty)\times \rr^d$, there exists $(q_0,p_0)$ such that $p_0 = d_{q_0} u_0$, $Q^t_0(q_0,p_0)=q$ and  \[g(t,q)=u_0(q_0)+ \aaa^t_0(\gamma),\]
where $\gamma$ denotes the Hamiltonian trajectory issued from $(q_0,p_0)$.

A family of operators $\left(R^t_s\right)_{s \leq t}$ mapping $\cc^{0,1}(\rr^d)$  into itself is called a \textit{variational operator} if it satisfies the following conditions:
\begin{enumerate}
	\item \label{monot} Monotonicity: if $u\leq v$ are Lipschitz on $\rr^d$, then $R^t_s u \leq R^t_s v$ on $\rr^d$ for each $s \leq t$,
	\item \label{addi} Additivity: if $u$ is Lipschitz on $\rr^d$ and $c\in\rr$, then $R^t_s (c+u) =c + R^t_s u$, 
	\item \label{var} Variational property:  for each $\cc^1$ Lipschitz function $u_s$, $q$ in $\rr^d$ and $s\leq t$, there exists $(q_s,p_s)$ such that $p_s = d_{q_s} u_s$, $Q^t_s(q_s,p_s)=q$ and  \[R^t_s u_s(q)=u_s(q_s)+ \aaa^t_s(\gamma),\]
	where  $\gamma$ denotes the Hamiltonian trajectory issued from $(q(s),p(s))=(q_s,p_s)$. \end{enumerate}

There may be more than one variational solutions associated with a Cauchy problem. We will see in Proposition \ref{varmin} that the monotonicity assumption made on the variational operator is a step towards more uniqueness.

\begin{rem}\label{contiop} If a family of operators $R$ satisfies  \eqref{monot} and \eqref{addi}, and if $u$ and $v$ are two Lipschitz functions on $\rr^d$ with bounded difference, then \begin{displaymath}
	\|R^t_s u -R^t_s v\|_\infty \leq \|u-v\|_\infty.\end{displaymath}
	As a consequence, for all $s\leq t$, $R^t_s$ is a weak contraction, and it is continuous for the uniform norm.
\end{rem}

\subsubsection*{Existence and local estimates}
The existence of such a variational operator is given by the method of Sikorav and Chaperon, see \cite{viterboX}. It is possible to obtain localized estimates on this family of variational operators that are also valid for the viscosity operator (in fact, they are obtained for the viscosity operator by a limit iterating process, see \cite{wei}). They are stated explicitly for integrable Hamiltonians in \cite{roos1}, Addendum $2.26$. 
\begin{prop}\label{locesti} There exists a family of variational operators $(R^t_{s,H})_H$ such that if $H(p)$ and $\tilde{H}(p)$ are two integrable Hamiltonians with bounded second derivatives, then for $0\leq s \leq t$ and $u$  $L$-Lipschitz,		\begin{itemize}\renewcommand{\labelitemi}{$\bullet$}
		\item $\|R^t_{s,\tilde{H}} u -R^t_{s,H}u\|_\infty \leq (t-s)\|\tilde{H}-H\|_{\bar{B}\!\left(0,L\right)},$
		\item $\|V^t_{s,\tilde{H}} u -V^t_{s,H}u\|_\infty \leq (t-s)\|\tilde{H}-H\|_{\bar{B}\!\left(0,L\right)}.$
	\end{itemize}
	where $\bar{B}\!\left(0,L\right)$ denotes the closed ball of radius $L$ centered in $0$ and $\|f\|_K:= sup_K |f|$.
\end{prop}

\subsection{Extension to nonsmooth initial data}
\subsubsection*{Lipschitz initial data}
We will denote by $\dd u(q)$ the \emph{Clarke derivative} of a function $u:\rr^d\to \rr$ at a point $q\in \rr^d$. If $u$ is Lipschitz, it is the convex envelop of the set of reachable derivatives: \[\dd u(q)= \mathrm{co}\left(\left\{\lim_{n\to \infty} d u(q_n), q_n \underset{n\to \infty}{\to} q, q_n\in dom(du)\right\}\right).\] It is the singleton $\{du(q)\}$ if $u$ is $\cc^1$ on a neighbourhood of $q$. Variational property \eqref{var} can be extended to Lipschitz initial condition with the help of this generalized derivative.

\begin{prop}\label{extvar} If $R^t_s$ is a variational operator, for each Lipschitz function $u_s$, $q$ in $\rr^d$ and $s\leq t$, there exists $(q_s,p_s)$ such that $p_s \in \dd_{q_s} u_s$, $Q^t_s(q_s,p_s)=q$ and if $\gamma$ denotes the Hamiltonian trajectory issued from $(q(s),p(s))=(q_s,p_s)$, \[R^t_s u_s(q)=u_s(q_s)+ \aaa^t_s(\gamma).\]
\end{prop}
The proof of this proposition can be found in \cite{mathz}, Proposition $1.22$.
%
%

If $u_0$ is a Lipschitz initial condition, the generalized wavefront associated with the Cauchy problem for $u_0$ is still denoted by $\mathcal{F}_{u_0}$ and defined by: \begin{displaymath} \tag{F'} \label{defwf2}
\mathcal{F}_{u_0}=\left\{\left(t,q,u_0(q_0)+ \aaa^t_0(\gamma)\right) \left|\begin{array}{c}
t\geq 0, q\in \rr^d,\\
p_0 \in \dd u_0(q_0),\\
Q^t_0(q_0,p_0)=q.
\end{array} \right\}\right.
\end{displaymath}
Proposition \ref{extvar} implies that a variational operator applied to $u_0$ gives a continuous section of the wavefront $\mathcal{F}_{u_0}$. We will still call variational solution a Lipschitz function whose graph is contained in the generalized wavefront.

\subsubsection*{Semiconcave initial data}
A function $u:\rr^d \to \rr$ is \emph{$B$-semiconcave} if $q \mapsto u(q)-\frac{B}{2}\|q\|^2$ is concave. The function $u$ is \emph{semiconcave} if there exists $B$ for which $u$ is $B$-semiconcave.

The following theorem states that if $u_0$ is a $B$-semiconcave function, a variational operator is given by the minimal section of the wavefront $\mathcal{F}_{u_0}$ for a duration depending only on $B$ and on the constant $C$ of Hypothesis \ref{esti} related to the Hamiltonian.

\begin{prop}\label{varmin}
	If $R^t_s$ is a variational operator and if $u_0$ is a Lipschitz $B$-semiconcave initial condition  for some $B>0$, then there exists $T>0$ depending only on $C$ and $B$ such that for all $(t,q)$ in $[0,T]\times \rr^d$, \begin{equation}
	R^t_0 u_0(q)=\inf\left\{S|(t,q,S)\in \mathcal{F}_{u_0}\right\}= \inf\left\{ u_0(q_0)+ \aaa^t_0(\gamma)\left|\begin{array}{c}
	(q_0,p_0)\in \rr^d\times  \rr^d, \\
	p_0 \in \dd u_0(q_0),\\
	Q^t_0(q_0,p_0)=q.
	\end{array} \right\}\right.
	\end{equation}
	where $\gamma$ denotes the Hamiltonian trajectory issued from $(q(0),p(0))=(q_0,p_0)$.
	
	Moreover if $H$ is integrable (\emph{i.e.} depends only on $p$), we can choose $T=1/BC$.
	\end{prop}

This theorem implies on one hand that for a semiconcave initial condition, the minimal section of the wavefront is continuous for small time. On the other hand, it yields that the variational operator gives in that case a variational solution which is less or equal than any other variational solutions on $[0,T]\times \rr^d$.

\begin{exe} In dimension $1$, if $u_0(q)=-|q|$ and if the Hamiltonian is integrable and has the shape of Figure \ref{fig} left, the wavefront at time $t$ has the shape of Figure \ref{fig} right and its minimal section, thickened on the figure, gives the value of $R^t_0 u_0$ above each point $q$.  On this example, there are five different variational solutions, but only the minimal one is given by a variational operator.
	\begin{figure}[h]
		\begin{minipage}{\textwidth}
			\begin{center}
				\def\svgwidth{0.40\columnwidth} 
				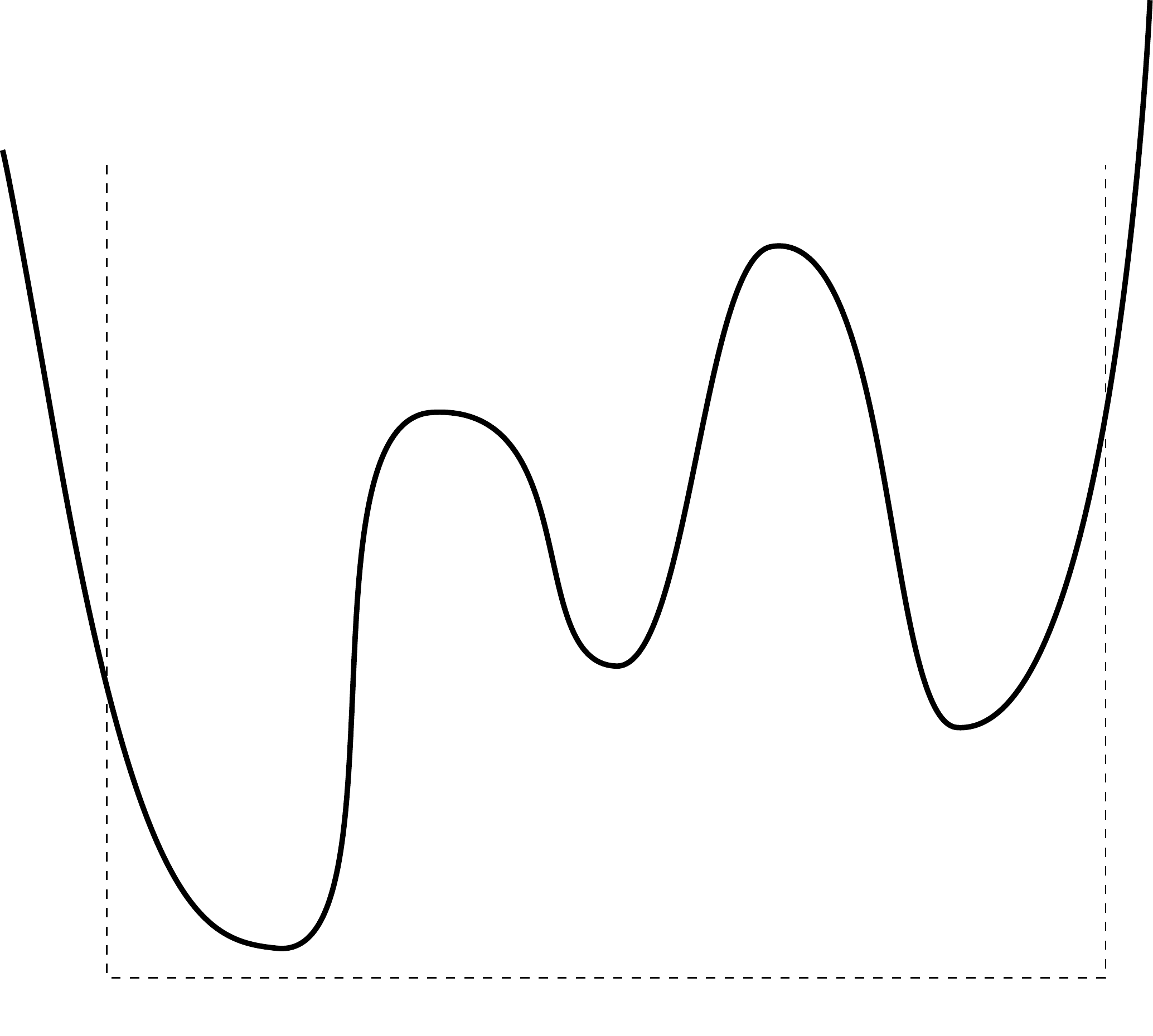 \hskip 1cm
				\def\svgwidth{0.50\columnwidth} 
				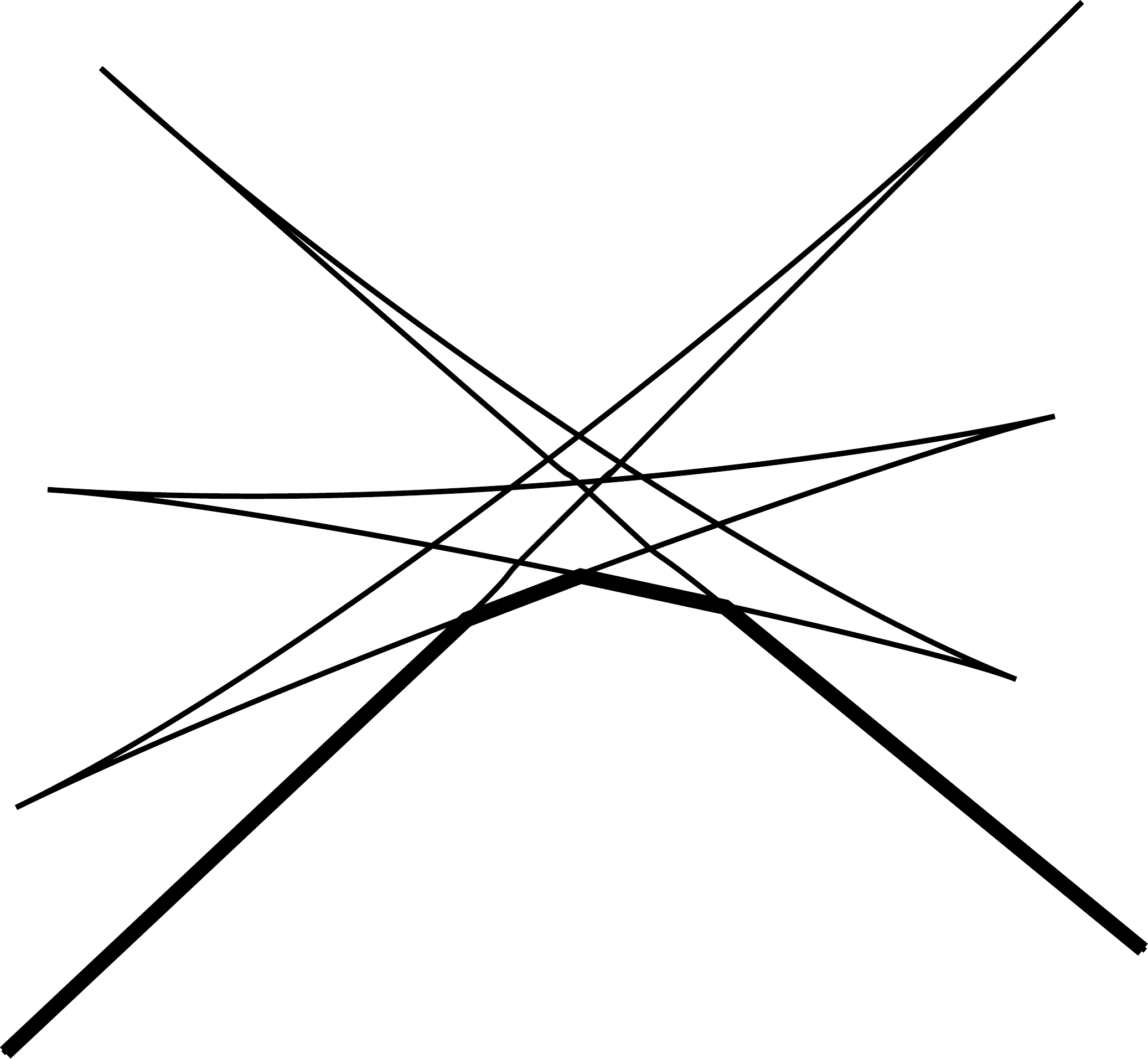				
			\end{center}
			\caption{Left: graph of $H$. Right: cross-section of the wavefront $\mathcal{F}_{u_0}$ at time $t$.}
			\label{fig}
		\end{minipage}
	\end{figure}
\end{exe}
An analogous argument to the one proving Proposition \ref{varmin} gives a first element of comparison between viscosity and variational solutions in the semiconcave framework. It is originally due to P. Bernard, see \cite{bernard2}.
\begin{prop}\label{v<v} Let $H$ be a Hamiltonian satisfying Hypothesis \ref{esti} with constant $C$.
	If $R^t_s$ is a variational operator and $u_0$ is a Lipschitz $B$-semiconcave initial condition for some $B>0$, then there exists $T>0$ depending only on $C$ and $B$ such that \[V^t_0 u_0\leq R^t_0 u_0.\]	
	for all $0\leq t \leq T$. Consequently, the viscosity solution is smaller than any variational solution on
	$[0,T] \times \rr^d$.
	
	Moreover if $H$ is integrable, we can choose $T=1/BC$.
\end{prop}

The article is organized as follows: Section \ref{proofvarmin} is independent from the rest, where we prove Propositions \ref{varmin} and \ref{v<v} for any Hamiltonian satisfying Hypothesis \ref{esti}. The rest of the article deals with integrable Hamiltonians: in Section \ref{proofjoukrec} we prove Corollary \ref{joukrecsc} which is a Lipschitz version of Theorem \ref{joukrec}. It is the corollary of Proposition \ref{joukrecsc1}, stated in terms of semiconcave initial conditions, which is proved by reduction to one or two-dimensional considerations, contained in Propositions \ref{joukrecdim1} and \ref{contrex}. In Section \ref{dim1} we study the case of the dimension $1$ and prove Proposition \ref{joukrecdim1}. 
In Section \ref{ce2} we study an example for the saddle Hamiltonian in dimension $2$ in order to prove Proposition \ref{contrex}. In Section \ref{u0C2} we deduce Theorem \ref{joukrec} from its Lipschitz counterpart Corollary \ref{joukrecsc} by approximation.

\section{Nonsmooth version of Theorem \ref{joukrec}} \label{proofjoukrec}
In this section we prove the following proposition, from which we deduce Corollary \ref{joukrecsc} which is the nonsmooth counterpart of Theorem \ref{joukrec}.
\begin{prop}\label{joukrecsc1}
	If $p\mapsto H(p)$ is a neither convex nor concave integrable Hamiltonian with second derivative bounded by $C$,  there exist $B>0$ and a Lipschitz $B$-semiconcave initial condition $u_0$ such that the variational solution given by the minimal section of the wavefront does not solve \eqref{HJ} in the viscosity sense at some point $(t,q)$ of $[0,1/BC]\times\rr^d$.
\end{prop}
\begin{cor}\label{joukrecsc}
	If  $p\mapsto H(p)$ is a neither convex nor concave integrable Hamiltonian with bounded second derivative, there exists a Lipschitz initial condition $u_0$ such that the graph of the viscosity solution associated with $u_0$ is not included in the wavefront $\ff_{u_0}$.
\end{cor}
\begin{proof}[Proof of Corollary \ref{joukrecsc}]
	Take a $B$-semiconcave initial condition $u_0$ as in Proposition \ref{joukrecsc1}. If $C$ is a bound on $d^2H$, Proposition \ref{varmin} states on one hand that the minimal section of the wavefront coincides with a variational solution on $[0,1/BC] \times \rr^d$, and on the other hand Proposition \ref{v<v} gives that on the same set, the viscosity solution associated with $u_0$ is less or equal than any variational solution. As a consequence the graph of the viscosity solution lies below the wavefront, and cannot coincide with the minimal section by Proposition \ref{joukrecsc1}. Hence there is a point of $[0,1/BC]\times\rr^d$ above which the graph of the viscosity solution lies strictly below the wavefront. 
\end{proof}

The outline of the proof of Proposition \ref{joukrecsc1} is the following: we give the statements in dimension $1$ (Proposition \ref{joukrecdim1}) and for $H(p_1,p_2)=p_1p_2$ (Proposition \ref{contrex}), and then reduce the situation to the first case or to an  approximation of the second case. Proposition \ref{ncnc} gives in that purpose a characterization of neither convex nor concave functions, and Proposition \ref{propred} deals with the effect on the variational and viscosity operators of an affine transformation or dimensional reduction of the Hamiltonian.

\begin{prop}[One-dimensional case]\label{joukrecdim1}
	If $H:\rr \to \rr$ is a neither convex nor concave integrable Hamiltonian with bounded second derivative, there exists $\delta>0$ and a semiconcave Lipschitz initial condition $u_0$ such that \[R^t_{0,H} u_0\neq V^t_{0,H} u_0 \,\;\,\; \forall t<\delta.\]
\end{prop}
Note that $\delta$ will be small enough so that $R^t_{0,H} u_0$ is uniquely defined, by Proposition \ref{varmin}.
This proposition is proved in \S \ref{contrex1}, and is really based on the example in dimension $1$ known at least since \cite{chenciner}. In contrast, the following two-dimensional example is the main novelty of this work.
\begin{prop}[Saddle Hamiltonian]\label{contrex} If $H(p_1,p_2)=p_1p_2$, for all $L>0$, there exists a $L$-Lipschitz, $L$-semiconcave initial condition $u_0$ such that   \[R^t_{0,H} u_0\neq V^t_{0,H} u_0 \,\;\,\; \forall t<1/2L.\]
\end{prop}
Note that $R^t_{0,H} u_0$ is uniquely defined when $t<1/2L$ by Proposition \ref{varmin}.
This proposition is proved in Section \ref{ce2}, where we explicit a suitable initial condition for which the wavefront has a single continuous section with a shock denying the entropy condition.

The following proposition makes precise the idea that a non-convex non-concave function is either a wave or a saddle. We will proceed further with the reduction of a one-dimensional non convex non concave function in Lemma \ref{tec}.
\begin{prop}\label{ncnc}
	A $\cc^2$ function $f:\rr^n \to \rr$ is neither convex nor concave if and only if it is neither convex nor concave along a straight line, or there exists $x$ in $\rr^n$ such that the Hessian $\hh f(x)$ admits both positive and negative eigenvalues. 
\end{prop}

\begin{proof} We denote by $S^+_n(\rr)$ (resp. by $S^-_n(\rr)$) the set of non-negative (resp. non-positive) symmetric matrices.
	
	Since a $\cc^2$ function is convex (resp. concave) if and only if its Hessian admits only non-negative (resp. non-positive) eigenvalues, it is enough to prove the following statement:
	if $f$ is a non-convex and non-concave $\cc^2$ function with $\hh f(x) \in S^+_n(\rr) \cup S^-_n(\rr)$ for all $x$, there exists a straight line along which $f$ is neither concave nor convex.
	
	Under the assumptions of this statement, the sets $U_1=\{x \in \rr^n | \hh f(x)\in S^-_n(\rr)\setminus \{0\} \}$ and $U_2=\{x \in \rr^n | \hh f(x)\in S^+_n(\rr)\setminus \{0\} \}$ are open and non empty: if $U_1$ is empty, $f$ is necessarily convex. If $x_1$ is in $U_1$, $\hh f(x_1)$ admits a negative eigenvalue. Hence for $x$ close enough to $x_1$, $\hh f(x)$ admits a negative eigenvalue and since $\hh f(x) \in S^+_n(\rr) \cup S^-_n(\rr)$ by hypothesis, necessarily $\hh f(x)$ is in $U_1$. We are going to apply the following lemma to the continuous function $A=\hh f$ and the sets $U_1$ and $U_2$.	
	\begin{lem}\label{cavexlem}
		If $A: \rr^n \to M_n(\rr)$ is a continuous function and $U_1$ and $U_2$ are two disjoint open sets on which $A$ does not vanish, there exists $(x_1,x_2)\in U_1 \times U_2$ such that\[ x_1-x_2 \notin {\rm Ker} A(x_1) \cup {\rm Ker} A(x_2).\]
	\end{lem}	
	
	Now, let us take $(x_1,x_2)$ in $U_1 \times U_2$ such that $x_1-x_2 \notin {\rm Ker} \hh f(x_1) \cup {\rm Ker} \hh f(x_2)$ and define $g(t)=f(tx_1 +(1-t)x_2)$. To show that the $\cc^2$ function $g$ is neither concave nor convex, we evaluate its second derivative:
	\[g''(t)= \hh f(tx_1+ (1-t)x_2)(x_1-x_2)\cdot (x_1-x_2).\]
	
	If $A$ is in $S^+_n(\rr)\cup S^-_n(\rr)$, $Ax\cdot x =0$ if and only if $Ax=0$. Since $\hh f(x_1)$ (resp. $\hh f(x_2)$) is in $S^-_n(\rr)$ (resp. $S^+_n(\rr)$), and 
	$x_1-x_2 \notin {\rm Ker} \hh f(x_1) \cup {\rm Ker} \hh f(x_2)$, we obtain $g''(1)=  \hh f(x_1)(x_1-x_2)\cdot (x_1-x_2) <0$ since $x_1-x_2$ is not in ${\rm Ker} \hh f(x_1)$, and $g''(0)=\hh f(x_2)(x_1-x_2)\cdot (x_1-x_2) >0$ since $x_1-x_2$ is not in ${\rm Ker} \hh f(x_2)$. Thus, $g$ is neither concave nor convex.
\end{proof}

\begin{proof}[Proof of Lemma \ref{cavexlem}.]
	For each $x_1^\circ \in U_1$, since $A(x_1^\circ)$ is a nonzero matrix, there exists $x_2^\circ$ in the open set $U_2$ such that $A(x_1^\circ)(x_1^\circ -x_2^\circ)\neq 0$. Since $(x_1,x_2)\mapsto A(x_1)(x_1-x_2)$ is continuous, we may assume up to a diminution of $U_1$ and $U_2$ that $A(x_1)(x_1-x_2)\neq 0$ for all $(x_1,x_2) \in U_1 \times U_2$.
	
	Now let us fix $x_2^\circ$ in $U_2$. Again, since $A(x_2^\circ)$ is nonzero, there exists $x_1^\circ$ in the open set $U_1$ such that $A(x_2^\circ)(x_1^\circ -x_2^\circ)\neq 0$, and the previous argument gives that $A(x_1^\circ)(x_1^\circ-x_2^\circ)\neq 0$, hence the conclusion.
\end{proof}

The next proposition deals with the behavior of the variational and viscosity operators when reducing or transforming the Hamiltonian. Let us first describe formally the effect of such transformations on the classical solutions.\\
\emph{Affine transformations.} Let $H$ be a Hamiltonian on $\rr^d$. Let $A$ be an invertible matrix, $b$ and $n$ be vectors of $\rr^d$, $\alpha$ a real value and $\lambda$ a nonzero real value, and define $\bar{H}(p)=\frac{1}{\lambda} H(Ap+b)+p\cdot n + \alpha$. If $u:\rr \times \rr^d\to \rr$ is $\cc^1$ and $v(t,q)=u(\lambda t,^t\!\!Aq+\lambda tn)+b\cdot q + \alpha \lambda t $, then for all $(t,q)$, \[
\dd_t u(\tilde{t},\tilde{q})+\bar{H}\left(\dd_qu(\tilde{t},\tilde{q})\right)=0 \iff \dd_t v(t,q)+H(\dd_qv(t,q))=0,
\]
with $(\tilde{t},\tilde{q})=(\lambda t, ^t\!\!Aq+\lambda tn)$.\\
\emph{Reduction.}  Assume that  $H$ is defined on $\rr^{d_1}\times \rr^{d_2}$. Let us fix $p_2$ in $\rr^{d_2}$ and define $\bar{H}(p_1)=H(p_1,p_2)$. If $u:\rr \times \rr^{d_1}\to \rr$ is $\cc^1$ and $v(t,q_1,q_2)=u(t,q_1)+p_2\cdot q_2$, then for all $(t,q_1,q_2)$\[
\dd_t u(t,q_1)+\bar{H}\left(\dd_{q_1}u(t,q_1)\right)=0 \iff \dd_t v(t,q_1,q_2)+H(\dd_{q_1} v(t,q_1,q_2),\dd_{q_2} v(t,q_1,q_2))=0.
\]
Let us translate this in terms of operators.
\begin{prop}\label{propred} Let $H$ be a $\cc^2$ Hamiltonian with second derivative bounded by $C$.
	\begin{enumerate}
		\item \label{afftrans}  \emph{Affine transformations.} Let $u_0$ be a Lipschitz $B$-semiconcave initial condition. If $\bar{H}(p)=\frac{1}{\lambda} H(Ap+b)+p\cdot n + \alpha$ and $v_0(q)=u_0(^t\!Aq)+ b\cdot q$, 
		\[V^t_{0,H} v_0(q)=V^{\lambda t}_{0,\bar{H}} u_0 (^t\!Aq+\lambda tn)+b\cdot q + \alpha \lambda t\]
		for all $(t,q)$ and
		\[R^t_{0,H} v_0(q)=R^{\lambda t}_{0,\bar{H}} u_0 (^t\!Aq+\lambda tn)+b\cdot q + \alpha \lambda t\]
		as long as $t<1/||A||^2 BC$. 
		\item \emph{Reduction.}\label{red} Assume that $H$ is defined on $\rr^{d_1}\times \rr^{d_2}$, fix $p_2$ in $\rr^{d_2}$ and define $\bar{H}(p_1)=H(p_1,p_2)$.
		If $u_0$ is a Lipschitz $B$-semiconcave function on $\rr^{d_1}$, and $v_0(q_1,q_2)=u_0(q_1)+p_2\cdot q_2$, then 	\[V^t_{0,H} v_0(q_1,q_2)=V^{t}_{0,\bar{H}} u_0 (q_1)+p_2\cdot q_2\]
		for all $(t,q_1,q_2)$ and
		\[R^t_{0,H} v_0(q_1,q_2)=R^{t}_{0,\bar{H}} u_0 (q_1)+p_2\cdot q_2,
		\]
		as long as $t<1/BC$. 
	\end{enumerate}
		\end{prop}


\begin{proof} The viscosity equality is obtained by applying the formal transformation or reduction on the test functions (see Definition \ref{defivisc}), and the variational equality is obtained for small time by applying Proposition \ref{varmin} with the domain of validity given for integrable Hamiltonians, which is the same for $(\bar{H},u_0)$ and $(H,v_0)$ in both cases:
	\item \emph{Affine transformations:}	since  $v_0$ is $B||A||^2$-semiconcave, the domain of validity for $(H,v_0)$ is at least $[0,1/\|A\|^2BC)$. But $\|d^2\bar{H}\|\leq C||A||^2/\lambda$, hence the domain of validity for $(\bar{H},u_0)$ is at least $[0,\lambda/\|A\|^2BC)$ and  $\lambda t$ is in this domain if $t < 1/\|A\|^2BC$.
	\item \emph{Reduction:} since  $\|d^2\bar{H}\|\leq C$ and $v_0$ is $B$-semiconcave, the domain of validity for both $(\bar{H},u_0)$ and $(H,v_0)$ is at least $[0,1/BC]$.
\end{proof}

\begin{proof}[Proof of Proposition \ref{joukrecsc1}] If $H$ is a neither convex nor concave integrable Hamiltonian, Proposition \ref{ncnc} states that there is either a straight line along which $H$ is neither convex nor concave, or a point $p_0$ such that the Hessian matrix $\hh H(p_0)$ has both a positive and a negative eigenvalue.\\

In the first case, applying an affine transformation on the vector space we may assume without loss of generality (see Proposition \ref{propred}-\eqref{afftrans}) that $p\in \rr\mapsto H(p,0,\cdots,0)$ is neither convex nor concave, and we denote by $\bar{H}(p)=H(p,0,\cdots,0)$ the reduced Hamiltonian. 
Applied to $\bar{H}$, Proposition \ref{joukrecdim1} gives a semiconcave initial condition $u_0$ such that $R^t_{0,\bar{H}} u_0 \neq V^t_{0,\bar{H}} u_0 $ for all $t<T$.
With Proposition \ref{propred}-\eqref{red}, we get from $u_0$ a semiconcave Lipschitz initial condition $v_0:\rr \times \rr^d\to \rr$ for which $R^t_{0,{H}} v_0 \neq V^t_{0,{H}} v_0 $ for all $t<T$.\\

In the second case,  we may assume that the point of interest is a (strict) saddle point at $0$: if $p_0$ denotes the point for which $\hh H(p_0)$ has both a positive and a negative eigenvalue, take $\tilde{H}(p)=H(p_0-p)+p\cdot \nabla H(p_0)-H(p_0)$ and apply Proposition \ref{propred}-\eqref{afftrans}.

Then, up to another linear transformation on the vector space, the Hamiltonian may even be taken as \[H(p_1,p_2,\cdots, p_d)=p_1p_2 + K(p_1,p_2,\cdots, p_d),\]
where $K$ is a $\cc^2$ Hamiltonian with partial derivatives with respect to $p_1$ and $p_2$ vanishing at the second order: \[K(0,\cdots,0)=0, \,\dd_{p_1} K(0,\cdots,0)=0,\, \dd_{p_2} K(0,\cdots,0)=0,\, \dd^2_{(p_1,p_2)} K(0,\cdots,0)=0.\] 
We denote by $\bar{H}$ (resp. $\bar{K}$) the reduced Hamiltonians such that  \[\bar{H}(p_1,p_2)=H(p_1,p_2,0,\cdots,0)=p_1p_2 + \bar{K}(p_1,p_2).\] 
We still denote by $C$ a bound on the second derivatives of $H$ and $\bar{H}$.

Now, we define \[\bar{H}_\ep(p_1,p_2)=\frac{1}{\ep^2} \bar{H}(\ep p_1,\ep p_2) =p_1p_2 +\frac{1}{\ep^2} \bar{K}(\ep p_1,\ep p_2)\] and \[\bar{H}_0(p_1,p_2)=p_1p_2.\]

We fix $L>0$ and take $u_0$ as in Proposition \ref{contrex}: for all $0<t<1/2L$, there exists a point $q_t$ such that $R^t_{0,\bar{H}_0} u_0(q_t) \neq V^t_{0,\bar{H}_0} u_0(q_t)$. Let us now fix $t$ in $(0,1/2L)$.

Proposition \ref{locesti} gives that
\[\|R^t_{0,\bar{H}_\ep} u_0(q_t)-R^t_{0,\bar{H}_0} u_0(q_t)\|\leq t \sup_{\|p\|\leq L} \frac{1}{\ep^2} \bar{K}(\ep p)   \]
and 
\[\|V^t_{0,\bar{H}_\ep} u_0(q_t)-V^t_{0,\bar{H}_0} u_0(q_t)\|\leq t \sup_{\|p\|\leq L} \frac{1}{\ep^2} \bar{K}(\ep p).   \]

Since $\bar{K}$ is zero until second order at $0$, $\frac{1}{\ep^2} \bar{K}(\ep p) = \circ(\|p\|^2)$ and $\sup_{\|p\|\leq L} \frac{1}{\ep^2} \bar{K}(\ep p)$ tends to $0$ when $\ep$ tends to $0$. Thus, there exists $\ep>0$ (depending on $t$) such that \[\sup_{\|p\|\leq L} \frac{1}{\ep^2} \bar{K}(\ep p)  < \frac{1}{3t}\left|R^t_{0,\bar{H}_0} u_0(q_t) - V^t_{0,\bar{H}_0} u_0(q_t)\right|,\]
and for such an $\ep$, we then have $R^t_{0,\bar{H}_\ep} u_0(q_t) \neq V^t_{0,\bar{H}_\ep} u_0(q_t)$.

Let us go back to $\bar{H}$, using Proposition \ref{propred}-\eqref{afftrans} with $\lambda=\ep^2$, $A=\ep {\rm id}$ and $n$, $b$ and $\alpha$ equal to zero. Defining $v_0(q)=u_0(\ep q)$, we get
\[V^{t/\ep^2}_{0,\bar{H}} v_0(q_t/\ep) =V^{t}_ {0,\bar{H}_\ep} u_0(q_t)  \]
and
\[R^{t/\ep^2}_{0,\bar{H}} v_0(q_t/\ep) = R^{t}_ {0,\bar{H}_\ep} u_0(q_t)  \]
as long as $\frac{t}{\ep^2}<\frac{1}{\ep^2LC}$ (which is the case since $C>2$ and $t<1/2L$), and as a consequence
\[V^{t/\ep^2}_{0,\bar{H}} v_0(q_t/\ep) \neq R^{t/\ep^2}_{0,\bar{H}} v_0(q_t/\ep).\]
Note that since $v_0$ is $\ep^2 L$-semiconcave, $t/\ep^2$ belongs to the domain of validity of Proposition \ref{varmin} which is here $(0,1/\ep^2LC)$. As in the previous case we get the semiconcave initial condition suiting the non reduced Hamiltonian $H$ via Proposition \ref{propred}-\eqref{red}.
\end{proof}

\section{One-dimensional integrable Hamiltonian}
 \label{dim1}

With the help of Lemma \ref{tec}, stated and proved at the end of this section, we reduce Proposition \ref{joukrecdim1}, the one-dimensional counterpart of Proposition \ref{joukrecsc1} (see \S\ref{proofjoukrec}), to the following statement, giving a situation where there is only one variational solution, that does not match with the viscosity solution. 
\begin{prop}\label{dim1ex} Let $H$ be a $\cc^2$ Hamiltonian with bounded second derivative such that $H(-1)=H(1)=H'(1)=0$, $H'(-1)<0$, $H''(1)<0$, and $H<0$ on $(-1,1)$.
	
	Then if $f$ is a $\cc^2$ Lipschitz function with $f(0)=f'(0)=0$, with bounded second derivative and strictly convex on $\rr_+$, 
	and $u_0(q)=-|q|+f(q)$, the unique variational solution $(t,q)\mapsto R^t_0 u_0(q)$ does not solve the Hamilton--Jacobi equation \eqref{HJ} in the viscosity sense for all $t$ small enough.
\end{prop}
With the vocabulary of Definition \ref{oleinik}, we work here on a specific case where the entropy condition is strictly satisfied between the  derivatives at $0$ of the initial condition, and the Lax condition is strictly satisfied on one side, and an equality on the other side, see Figure \ref{exdim1} left.

The proof consists in showing that under the assumptions of Proposition \ref{dim1ex}, when $t$ is small enough, the wavefront at time $t$ presents a unique continuous section, with a shock that denies Oleinik's entropy condition (see Proposition \ref{RH}).

\subsection{Proof of Proposition \ref{dim1ex}}\label{contrex1}
Let us fix the notations for the parametrization that follows directly from the wavefront definition (see \eqref{defwf2}).
Since $u_0$ is differentiable on $\rr\setminus \{0\}$, its Clarke derivative is reduced to a point outside zero and is the segment $[-1,1]$ at zero. The wavefront is hence the union of three pieces $\mathcal{F}_t^\ell$, $\mathcal{F}_t^r$ and $\mathcal{F}_t^0$ respectively issued from the left part, the right part, and the singularity of the initial condition, with the following parametrizations:
\[\begin{array}{lc}
\mathcal{F}_t^\ell : \left\{\begin{array}{l}
q+tH'(u_0'(q)),\\
u_0 (q)+tu_0'(q)H'(u_0'(q)) -t H(u_0'(q)),
\end{array}\right.& q<0,\\

\mathcal{F}_t^r : \left\{\begin{array}{l}
q+tH'(u_0'(q)),\\
u_0 (q)+tu_0'(q)H'(u_0'(q)) -t H(u_0'(q)),
\end{array}\right.& q>0,\\

\mathcal{F}_t^0 : \left\{\begin{array}{l}
tH'(p),\\
t\left(pH'(p)-H(p)\right),
\end{array}\right. & p \in [-1,1].
\end{array}\]

The structure of the wavefront for small time is adressed by Lemma \ref{unisec}. Figure \ref{exdim1} presents an example of the situation. 
\begin{lem}\label{unisec}
	Under the assumptions of Proposition \ref{dim1ex}, there exists $\delta>0$ such that for all $0<t<\delta$, the wavefront $\ff_t$ has a unique continuous section, presenting a shock between $\ff_t^0$ and $\ff_t^r$.
\end{lem}

With the previous parametrization, we may easily compute the slopes and convexity of the wavefront. We still denote by $C$ and $B$ the bounds on the second derivatives of $H$ and $u_0$.
\begin{prop} 	
	\begin{enumerate}
	\item {Slopes on the wavefront.} If $H''(p)\neq 0$ and $t>0$, the slope of $\ff^0_t$ at the point of parameter $p$ is $p$. If $t<1/BC$, the slope of $\ff_t^r$ at the point of parameter $q$ is $u_0'(q)$.
	
	\item {Convexity of the right arm.} 
		If $u_0$ is convex (resp. concave) on $[0,\delta]$, then for $t < 1/BC$, the portion of $\ff^r_t$ parametrized by $q \in (0,\delta]$ is convex (resp. concave). 
	\end{enumerate}	
	\label{slopes}
\end{prop}

\begin{proof}
	\begin{enumerate}
	\item If $(x(u),y(u))$ is the parametrization of a curve, the slope at the point of parameter $u$ is given by $y'(u)/x'(u)$ when $x'(u)$ is nonzero. For $\ff^0_t$, it comes $x'(p)=tH''(p)$ and $y'(p)=p x'(p)$, which proves the statement. For $\ff^r_t$, if $t<1/BC$, $x'(q)=1+t u_0''(q)H''(u_0'(q))>0$ since $u_0''$ and $H''$ are respectively bounded by $B$ and $C$, and the statement results from $y'(q)=u_0'(q)x'(q)$. 
	\item The convexity of $\mathcal{F}^r_t$ at a point of parameter $q$ is given by the sign of the ratio $\frac{x'(q)y''(q)-x''(q)y'(q)}{x'(q)^3}$. For $t<1/BC$, $x'(q)>0$ and as $y'(q)=u_0'(q)x'(q)$, \[
	\frac{x'(q)y''(q)-x''(q)y'(q)}{x'(q)^3} =\frac{x'\left( u_0''x'+u_0'x''\right)-x''u_0'x'}{x'^3}=  \frac{ u_0''(q)}{x'(q)},
	\]
	which proves the statement.
	\end{enumerate}
\end{proof}

The fact that $\ff_t^0$ depends homothetically on $t$ suggests to look for each $t>0$ at the homothetic reduction of the wavefront at time $t$, where both coordinates are divided by $t$. We call it \emph{reduced wavefront} and denote it by $\tilde{\ff}_t$. It admits the following parametrizations, where $q=tx$:
\[\begin{array}{lc}
\tilde{\ff}_t^\ell : \left\{\begin{array}{l}
x+H'(u_0'(tx)),\\
\frac{u_0(tx)}{t} +u_0'(tx)H'(u_0'(tx))- H(u_0'(tx)),
\end{array}\right.& x<0,\\

\tilde{\ff}_t^r : \left\{\begin{array}{l}
x+H'(u_0'(tx)),\\
\frac{u_0(tx)}{t}+u_0'(tx)H'(u_0'(tx)) - H(u_0'(tx)),
\end{array}\right.& x>0,\\

\tilde{\ff}_t^0 : \left\{\begin{array}{l}
H'(p),\\
pH'(p)-H(p),
\end{array}\right. & p \in [-1,1].
\end{array}\]
The asset of the reduced wavefront is that it admits a nontrivial limit when $t$ tends to $0$. The piece issued from the singularity $\ff^0=\tilde{\ff}_t^0$ does not depend on $t$, while $\tilde{\ff}^r_t$ and $\tilde{\ff}^\ell_t$ converge to straight half-lines denoted by $\ff^r$ and $\ff^\ell$. These half-lines coincide respectively with $\tilde{\ff}^r_t$ and $\tilde{\ff}^\ell_t$ at their fixed endpoints, see $t \ff^r$ and $\ff^r_t$ on Figure \ref{exdim1}. A consequence of Proposition \ref{slopes} is that $\tilde{\ff}^\ell_t$ is a graph as long as $t<1/BC$, and the same applies to $\tilde{\ff}^r_t$.

\begin{figure}[h]
	\begin{center}
		\def\svgwidth{0.35\columnwidth} 
		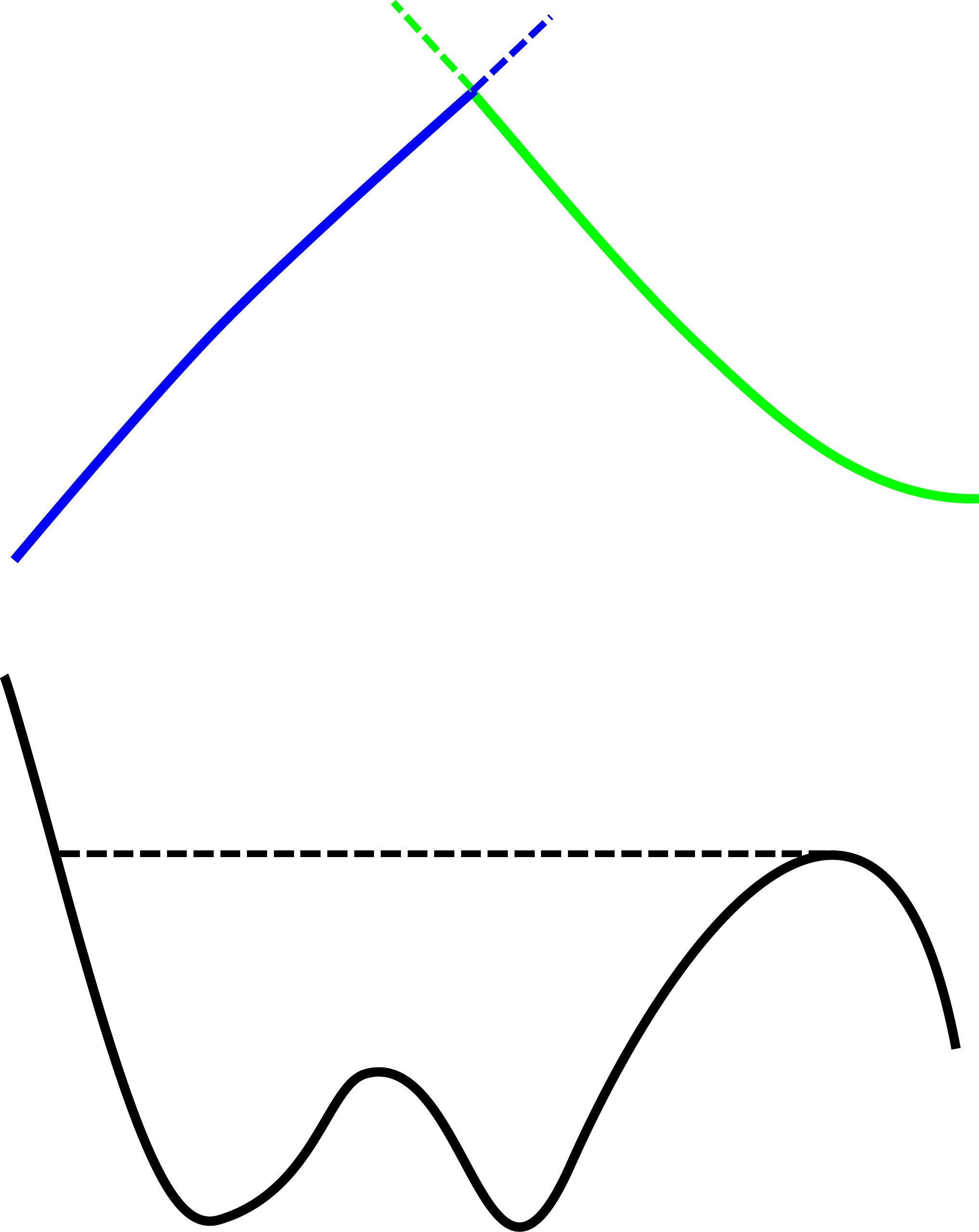 
		\hskip .3cm
		\def\svgwidth{0.55\columnwidth} 
		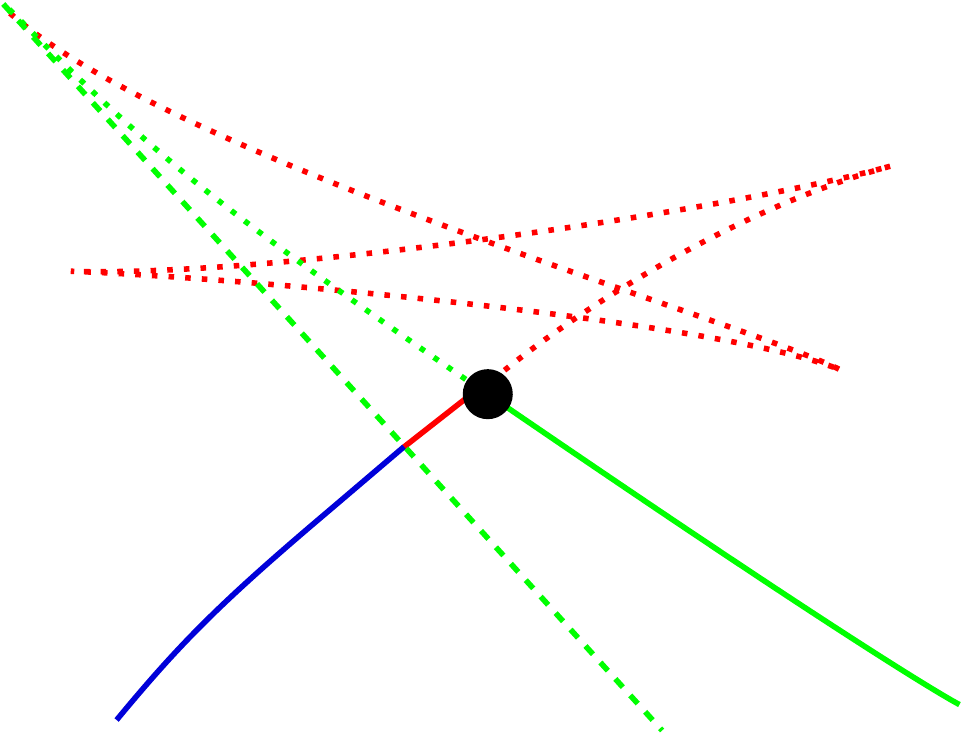 
	\end{center}
	\caption{The variational solution, given by the unique continuous section of the wavefront, does not solve the \eqref{HJ} equation in the viscosity sense at the dot.}
	\label{exdim1}
\end{figure}


\begin{proof}[Proof of Lemma \ref{unisec}] It is enough to prove the result for the reduced wavefront $\tilde{\ff}_t$, where both coordinates are divided by $t$.
 Using the left and right derivatives of $u_0$ and the fact that $H(1)=H(-1)=H'(1)=0$, we write the parametrization of the limit of the reduced wavefront:
	\[	\begin{array}{lc}
	
	\mathcal{F}^\ell : \left\{\begin{array}{l}
	x,\\
	x,
	\end{array}\right.& x<0,\\
	
\mathcal{F}^r : \left\{\begin{array}{l}
	x+H'(-1),\\
	-x-H'(-1),
	\end{array}\right.& x>0,\\
	
\mathcal{F}^0 : \left\{\begin{array}{l}
	H'(p),\\
	pH'(p)-H(p),
	\end{array}\right. & p \in [-1,1].
	\end{array}\]

	The left and right arms of the limit front are respectively the graph of $-{\rm id}$ and ${\rm id}$ on $(-\infty,0)$ and on $(H'(-1),\infty)$, where $H'(-1)<0$. The assumption $H<0$ on $(-1,1)$ implies that for all $p$ in $(-1,1)$,
	\begin{equation}\label{eq} pH'(p)-H(p)>-|H'(p)|,\end{equation}
	and this inequality is also satisfied for $p=-1$ since $H(-1)=0$ and $H'(-1)<0$. The unique continuous section of the limit front is hence the graph of $x\mapsto -|x|$. It presents a shock at $(0,0)$, which belongs to $\ff^r$ and $\ff^0$ respectively with parameters $x=-H'(-1)>0$ and $p=1$. Furthermore, \eqref{eq} implies that this shock is not a double point of $\ff^0$.

	Proposition \ref{slopes} states that since $f$ is strictly convex on $\rr_+$, $\ff^r_t$ and hence $\tilde{\ff}^r_t$ are strictly convex curves for all $t>0$. Looking at the slope for a parameter $x\to 0$ shows that $\tilde{\ff}^r_t$ admits the right arm of the limit front, $\ff^r$, as a tangent at its endpoint. Since $\tilde{\ff}^r_t$ is convex, it is hence positioned strictly above $\ff^r$. Since $\tilde{\ff}^\ell_t$ is for all $t<1/BC$ a graph with fixed endpoint at $(0,0)$, we may focus on what happens on the half-plane situated over the second diagonal.  	
	
	As $H''(1)<0$, there exists $\eta>0$ such that $H''<0$ on $(1-\eta,1]$, and the piece of $\ff^0$ parametrized by $p \in (1-\eta,1]$, denoted $\ff^0_{(1-\eta,1]}$, is immersed. Since $\ff^0$ is compact, we may assume up to taking a smaller $\eta$ that $\ff^0_{(1-\eta,1]}$ does not contain any double point either. 	
With this choice of $\eta$, the intersection $\ff^r \cap \ff^0_{(1-\eta,1]}$ is exactly the point $(0,0)$ and is transverse, since the slopes at the shock are $-1$ and $1$.

	Let us denote the family of parametrizations of $\tilde{\ff}^t_r   \cup 	\ff^r $ by 	\[g^r(t,x)= \left\{
		\begin{array}{lr}
		\left(x+H'(u_0'(tx)),\frac{u_0(tx)}{t} +u_0'(tx)H'(u_0'(tx))- H(u_0'(tx)\right)& \textrm{ if } t\neq 0,\\
		\left(x+H'(-1),-x-H'(-1)\right) & \textrm{ if } t= 0,
		\end{array}
		\right.    \]
The function $t \mapsto g^r(t,\cdot)$ is continuous on $[0,\infty)$ in the $\cc^1$-topology since the function $(t,x) \mapsto \left\{\begin{array}{l}
	u_0(tx)/t   \text{ if } t>0\\
	- x \text{ if } t=0
	\end{array}
	\right.$ is $\cc^1$ on $[0,\infty) \times [0,\infty)$. The transverse intersection hence persists by the implicit function theorem in an intersection between $\tilde{\ff}^r_t$ and $\ff^0_{(1-\eta,1]}$, since $\tilde{\ff}^r_t$ is contained in the half-plane situated over the second diagonal.
	
%

	There is no other continuous section in $\tilde{\ff}_t$: for small time $t$, $\tilde{\ff}^r_t$ and $\tilde{\ff}^\ell_t$ do not cross and do not present double points; the existence of a second continuous section would then imply the existence of an intersection between $\ff^0$ and the part of $\tilde{\ff}^r_t$ at the right of the shock, or an intersection between $\ff^0$ and $\tilde{\ff}^\ell_t$, and neither can exist, by continuity.
\end{proof}
It is now enough to prove that the obtained shock denies the Lax condition.
\begin{proof}[Proof of Proposition \ref{dim1ex}] For all $t$, the graph of a variational solution is included in the wavefront $\ff_t$. Lemma \ref{unisec} gives a $\delta>0$ for which every $\ff_t$ has a unique continuous section if $t\leq \delta$, which implies that the variational solution is given by this section for small $t$. 
	Lemma \ref{unisec} states also that this section presents a shock between  $\ff^0_t$ and ${\ff}^r_t$. 
	
	Let us prove that Lax condition is violated at this shock. A fortiori, Oleinik’s entropy condition is violated, which by Proposition \ref{RH} will imply that the variational solution is not a viscosity solution. For all $t$ in $(0,\delta)$, the shock is given by parameters
	$(q_t,p_t)$,
	such that $q_t>0$, $p_t\in [-1,1]$ and 
	\[
	\left\{\begin{array}{rcl}
	q_t+t H'\left(u_0'(q_t)\right)&=&t H'(p_t),\\
	{u_0(q_t)}+t u_0'(q_t)H'\left(u_0'(q_t)\right) - tH\left(u_0'(q_t)\right)&=&t p_tH'(p_t)-t H(p_t),
	\end{array}\right.\]
	Injecting the first equation  multiplied by $u_0'(q_t)$ into the second gives, after reorganization:
	\[t\left(H(p_t)-H(u_0'(q_t) - (p_t-u_0'(q_t))H'(p_t)\right) =q_t u_0'(q_t)- {u_0(q_t)}. \]
	The linear part of $u_0$ cancels in the right hand side, which equals $q_t f'(q_t)- {f(q_t)}$. The strict convexity of $f$ implies that $f'(h)>f(h)/h$ for all $h>0$, hence the right hand side is positive for $t>0$. As a consequence, for $t$ in $(0,\delta)$, 
	\[{H(p_t)-H(u_0'(q_t))} >\left(p_t-u_0'(q_t)\right) H'(p_t).\]
	By Proposition \ref{slopes}, the slopes at the shock are $u_0'(q_t)$ and $p_t$. Comparing with Definition \ref{oleinik}, this inequality hence reads as the opposite of the Lax condition, hence Oleinik's entropy condition is violated for the shock presented by the variational solution for $t<\delta$, and the conclusion holds.
\end{proof}

\subsection{Proof of Proposition \ref{joukrecdim1}}
The idea behind Lemma \ref{tec} is that for any non convex non concave Hamiltonian in dimension $1$, there is a frame of variables over which the Hamiltonian looks like Figure \ref{exdim1}, left.

\begin{lem}\label{tec} If $H:\rr \to \rr$ is a $\cc^2$ neither convex nor concave Hamiltonian, up to a change of function $p\mapsto H(-p)$, there exist $p_1<p_2$ such that $H''(p_2)<0$, and
	\begin{align}
	\forall p \in (p_1,p_2),\,\; \frac{H(p)-H(p_1)}{p -p_1} & <  \frac{H(p_2)-H(p_1)}{p_2 -p_1}, \label{teceq1} \\
	H'(p_1)& <\frac{H(p_2)-H(p_1)}{p_2 -p_1}=H'(p_2).	 \label{teceq2}
	\end{align}	\end{lem}
	In terms of Definition \ref{oleinik}, \eqref{teceq1} means that the entropy condition is strictly satisfied between $p_1$ and $p_2$, and \eqref{teceq2} that the Lax condition is an equality at $p_2$, and an inequality at $p_1$. We are now just one affine step away from the hypotheses of Proposition \ref{dim1ex}.
			

\begin{proof}
	 	If $H:\rr \to \rr$ is neither convex nor concave, there exist in particular $p_1^\circ$ and $p_2^\circ$ such that $H''(p_1^\circ)>0$ and $H''(p_2^\circ)<0$, and we may assume up to the change of Hamiltonian $p \mapsto H(-p)$ that $p_1^\circ<p_2^\circ$.	 \\	
	 	\textit{Sketch of proof.} The proof consists in choosing adequate $p_1$ and $p_2$, which will be done differently depending on the entropy condition between $p_1^\circ$ and $p_2^\circ$ being satisfied or not. An impatient reader could be satisfied by the choice of $p_1$ and $p_2$ suggested in Figure \ref{fig2}. If the entropy condition is denied, we take $p_1=p_1^\circ$ and $p_2$ such that the slope of the cord joining $p_1$ and $p_2$ is maximal. We then need to slightly perturb $p_1$ in order to get the condition $H''(p_2)<0$. If the entropy condition is satisfied, we take $p_2=p_2^\circ$  and $p_1$ is given by the last (before $p_2$) intersection between the tangent at $p_2$ and the graph of $H$. Again, a perturbation will be done to ensure that $H'(p_1)<H'(p_2)$.
	 	\begin{figure}[h]
	 		\begin{minipage}{\textwidth}
	 			\begin{center}
	 				\def\svgwidth{0.9\columnwidth} 
	 				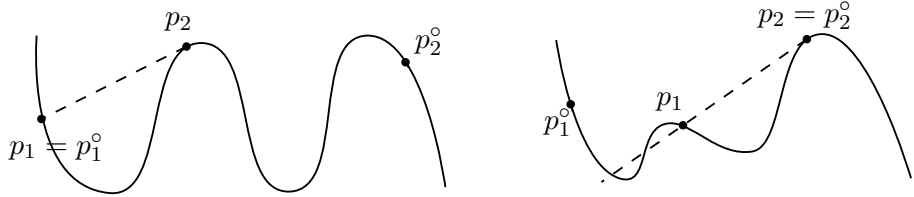 		
	 			\end{center}
	 			\caption{Both figures present a graph of $H$ with a dashed tangent at $p_2$. Left: entropy condition denied between $p_1^\circ$ and $p_2^\circ$. Right: entropy condition satisfied between $p_1^\circ$ and $p_2^\circ$.}
	 			\label{fig2}
	 		\end{minipage}
	 	\end{figure}	
	 	
	 	\begin{itemize}
	 	\item 	If the entropy condition is denied between $p_1^\circ$ and $p_2^\circ$, we define $p_1=p_1^\circ$ and
	 	\[p_2=\inf\left\{p\in(p_1,p_2^\circ), \;\,\frac{H(p)-H(p_1)}{p-p_1}=\sup_{\tilde{p}\in (p_1,p_2^\circ]}\frac{H(\tilde{p})-H(p_1)}{\tilde{p}-p_1} \right\}.\]
	 	Let us show that these quantities are well defined, and prove \eqref{teceq1} and \eqref{teceq2}.
	 	
	 	The function $f:p\mapsto \frac{H(p)-H(p_1)}{p-p_1}$ may be extended continuously at $p_1$ by $H'(p_1)$, hence it reaches a maximum $M$ on $[p_1,p_2^\circ]$. It cannot be attained at $p_1$, or else the Taylor expansion of $\frac{H(p)-H(p_1)}{p-p_1} \leq  H'(p_1)$ gives that $H''(p_1)\leq 0$, which is excluded. As a consequence $M>H'(p_1)$. It cannot be attained at $p_2^\circ$ because $\frac{H(p)-H(p_1)}{p-p_1} \leq \frac{H(p_2^\circ)-H(p_1)}{p_2^\circ-p_1}$ for all $p$ in $[p_1,p_2^\circ)$ if and only if the entropy condition is satisfied between $p_1$ and $p_2^\circ$, which is excluded.
	 	We hence proved that the supremum is attained on $(p_1,p_2^\circ)$. The infimum hence exists and belongs to $[p_1,p_2^\circ)$. By continuity of $f$, $f(p_2)=M$. This implies that $p_2 > p_1$ since $f(p_1)=H'(p_1)<M$, hence the infimum is a minimum. The equality \eqref{teceq1} follows directly from the definition of $p_2$.	
	 	 	
	 	Since $p_2$ is in $(p_1,p_2^\circ)$ and maximizes $f$, it is a critical point of $f$, which gives  $H'(p_2)=\frac{H(p_2)-H(p_1)}{p_2-p_1}=M$. Since $H'(p_1)<M$, \eqref{teceq2} is proved.\\
	 	
	 	Since $p_2$ maximizes $f$, $f''(p_2)\leq 0$ and as a consequence $H''(p_2)\leq 0$. In order to get $H''(p_2)<0$, let us prove that if $p_2^\circ$ is fixed, $p_1 \mapsto H'(p_2)$ is increasing in a neighbourhood of $p_1$.
	 	
	 	For $\ep>0$ small enough, $p_1+\ep <p_2$, $H''(p_1+\ep)>0$ and the entropy condition is denied between $p_1+\ep$ and $p_2^\circ$. We denote by $p_{2,\ep}$ the quantity associated with $p_1+\ep$ and $p_2^\circ$ as previously.
	 	
	 	On one hand, by definition of $p_2$, the entropy condition is strictly satisfied between $p_1$ and $p_2$, and in particular since $p_1+\ep$ is in $(p_1,p_2)$, \[
	 	\frac{H(p_2)-H(p_1+\ep)}{p_2-(p_1+\ep) } >\frac{H(p_2)-H(p_1)}{p_2-p_1}= H'(p_2).
	 	\]
	 	On the other hand, the previous work applied to $p_{2,\ep}$ gives that \[H'(p_{2,\ep})= \max_{p\in (p_1+\ep,p_2^\circ]} \frac{H(p)-H(p_1+\ep)}{p-(p_1+\ep) } \geq \frac{H(p_2)-H(p_1+\ep)}{p_2-(p_1+\ep) },\]
	 	and the two inequalities combined give that $H'(p_{2,\ep})>H'(p_2)$.
	 	
	 	Since $p_1 \mapsto H'(p_2)$ is increasing in a neighbourhood of $p_1$, using the Sard's theorem, we may assume without loss of generality that ${H}'(p_2)$ is a regular value of ${H}'$, up to a perturbation of $p_1$ within the open set $\{H''>0\}$. As a consequence, ${H}''(p_2)<0$, and the couple $(p_1,p_2)$ satisfies Lemma \ref{tec}.\\
	 	 
	 	 \item 	If the entropy condition is satisfied between $p_1^\circ$ and $p_2^\circ$, we define $p_2=p_2^\circ$ and
	 	 \begin{equation} \label{p_1} p_1= \sup \left\{p_1^\circ \leq p\leq p_2\, \left|\, \frac{H(p_2)-H(p)}{p_2-p} =H'(p_2)\right.  \right\}.\end{equation}
	 	 As $H''(p_2)$ is negative, the graph of $H$ is situated strictly under its tangent at $p_2$ over a neighbourhood of $p_2$, hence $ \frac{H(p_2)-H(p)}{p_2-p} >H'(p_2)$ on this neighbourhood. The entropy condition satisfied between $p_1^\circ$ and $p_2$ implies the Lax condition $\frac{H(p_2)-H(p_1^\circ)}{p_2-p_1^\circ} \geq H'(p_2)$. By the mean value theorem, the considered set is non empty and its supremum belongs to $[p_1^\circ,p_2)$, and by continuity of $p\mapsto  \frac{H(p_2)-H(p)}{p_2-p}$, we have that $\frac{H(p_2)-H(p_1)}{p_2-p_1} = H'(p_2)$. The entropy condition is strictly satisfied between $p_1$ and $p_2$ by maximality of $p_1$. The mean value theorem and the maximality of $p_1$ makes it clear that $H'(p_1)\leq H'(p_2)$ and that if $H'(p_1)=H'(p_2)$, $H''(p_1)\leq 0$. Let us prove that up to a perturbation we can assume $H'(p_1)<H'(p_2)$. 
	 	 
	 	Let us hence assume that $H'(p_1)=H'(p_2)$. First, by Sard's theorem, up to a perturbation of $p_2^\circ$, we may assume that $H'(p_2^\circ)$ is not a critical value of $H'$, which ensures since $H'(p_1)=H'(p_2^\circ)$ that $H''(p_1)$ is nonzero, hence negative (note that the sign of $H''(p_1^\circ)$ had no influence in the previous paragraph). We denote $p_1^\circ=p_1$ and look at the previous construction for this $p_1^\circ$ fixed and for a new $p_2$ close to $p_2^\circ$.  Without loss of generality we suppose that $H'(p_2^\circ)=H'(p_1^\circ)=H(p_2^\circ)=H(p_1^\circ)=0$. Since $H''(p_1^\circ)$ and $H''(p_2^\circ)$ are negative, there exists $\delta$ such that $H''$ is negative on $[p_1^\circ,p_1^\circ+\delta]\cup [p_2^\circ - \delta,p_2^\circ]$. By compacity, $H$ admits a maximum on $[p_1^\circ + \delta,p_2^\circ-\delta]$ which is negative, since the entropy condition is strictly satisfied between $p_1^\circ$ and $p_2^\circ$.
	 	
	 	Since $H'$ is decreasing on $[p_2^\circ-\delta,p_2^\circ]$, there exists $p_2 \in [p_2^\circ-\delta,p_2^\circ]$ such that $0<H'(p_2)< -\frac{m/2}{p_2^\circ-p_1^\circ}$. For such a $p_2$, the tangent of the graph of $H$ at $p_2$ lies strictly below the graph of $H$ over $[p_1^\circ+\delta,p_2^\circ-\delta]$ by definition of $m$, and also over $[p_2^\circ-\delta,p_2^\circ]$ by concavity of $H$. Equation \eqref{p_1} then defines a $p_1$ which is necessarily in $(p_1^\circ,p_1^\circ +\delta]$: as $H(p_1^\circ)=0$ and $H(p_2)<0$, the point $(p_1^\circ,H(p_1^\circ))$ is situated over the tangent of the graph of $H$ at $p_2$ which has a positive slope $H'(p_2)$. By concavity of $H$ on $[p_1^\circ,p_1^\circ+\delta]$, $H'(p_1)<H'(p_1^\circ)=0$, and as a consequence $H'(p_1)<H'(p_2)$.	 	
	 	The previous work proves that all the conditions of the proposition are then gathered for $p_1$ and $p_2$.		
	 	\end{itemize}
\end{proof}
We may now prove Proposition \ref{joukrecdim1}, joining Lemma \ref{tec} and Proposition \ref{dim1ex}.
\begin{proof}[Proof of Proposition \ref{joukrecdim1}.]
Let $H$ be a non convex non concave Hamiltonian with bounded second derivative. Using Proposition \ref{propred}-\eqref{afftrans} with $A=-{\rm id}$, we may apply Lemma \ref{tec} up to the change of function $p\mapsto H(-p)$. It gives $p_1<p_2$ such that $H''(p_2)< 0$, and  $ H'(p_1)<\frac{H(p)-H(p_1)}{p -p_1}<H'(p_2)=\frac{H(p_2)-H(p_1)}{p_2 -p_1}$ for all $p$ in $(p_1,p_2)$. 
	We  define \[
	\tilde{H}(p)= H(p)-H(p_2)- (p-p_2) H'(p_2),
	\]
	so that $\tilde{H}(p_2)=\tilde{H}'(p_2)=\tilde{H}(p_1)=0$. Note also that $\tilde{H}'(p_1)=H'(p_1)-H'(p_2)<0$.	
	The second order derivatives as well as the entropy condition are preserved by this transformation:  $\tilde{H}''(p_2)=H''(p_2)< 0$, and $\tilde{H}<0$ on $(p_1,p_2)$. 
	
	At last, we take the affine transformation $\phi:\rr \to \rr$ such that $\phi(-1)=p_1$ and $\phi(1)=p_2$ and define \[
	\bar{H}(p)=\tilde{H}(\phi(p)),
	\]
	so that $\bar{H}$ satisfies the assumptions of Proposition \ref{dim1ex}: $\bar{H}(-1)=\bar{H}(1)=\bar{H}'(1)=0$, $\bar{H}'(-1)<0$ since $\phi'>0$, $\bar{H}''(1)<0$ and $\bar{H}<0$ on $(-1,1)$. Proposition \ref{dim1ex} then gives a Lipschitz semiconcave
	initial condition $\bar{u}_0$ such that the variational solution denies the \eqref{HJ} equation associated with $\bar{H}$ for all $t$ small enough. Proposition \ref{propred}-\eqref{afftrans} applied to the two successive transformations gives then a Lipschitz semiconcave initial condition $u_0$, with right and left derivatives at $0$ respectively equal to $p_1$ and $p_2$, such that the variational solution denies the \eqref{HJ} equation associated with $H$ for all $t$ small enough.
	
	\end{proof}

\section{Example for the saddle Hamiltonian: proof of Proposition \ref{contrex}}\label{ce2}

In this paragraph we assume that $H(p_1,p_2)=p_1p_2$, with $(p_1,p_2)\in \rr^2$, and prove Proposition \ref{contrex} by presenting a suitable initial condition.



We choose an initial condition that coincide with the piecewise quadratic function $u(q_1,q_2)=\min\left(a(q_1^2-q_2),b(q_1^2-q_2)\right)$ on a large enough subset while being Lipschitz and semiconcave. We explicit the value of the variational solution for this initial condition on a large enough subset.
	\begin{prop} \label{exdim2} 
		Let $f: \rr \to \rr$ be a compactly supported $\cc^2$ function coinciding with $x \mapsto x^2$ on $[-1,1]$. 
		
		Let $u(q_1,q_2)= \min \left(a(f(q_1)-q_2),b(f(q_1)-q_2)\right)$ with $b>a>0$.
		
		Then if $-1 \leq   q_1\leq -\frac{3b}{2}t$, \[R^t_0 u(q_1,q_2)= \min\left(a((q_1+at)^2-q_2),b((q_1+bt)^2-q_2)\right).\]\end{prop}

		\begin{figure}[h!]
			\begin{center}
				\def\svgwidth{\columnwidth} 
				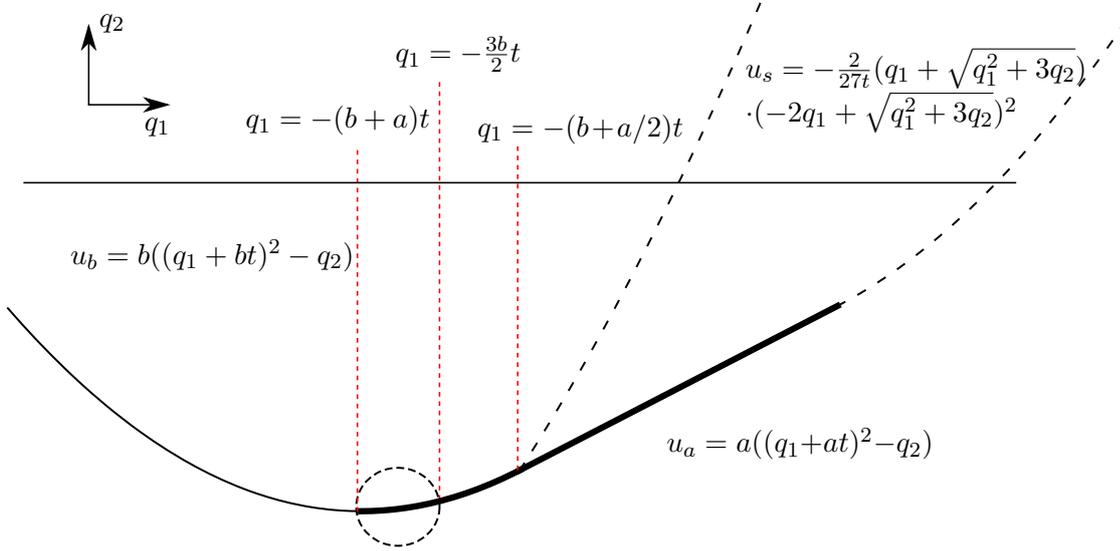 
			\end{center}
			\caption{Value of the variational solution associated with $u$ at time $t$, here for $a=1$, $b=2$ and $t=1/10$.}
			\label{varsolex}
		\end{figure}
		
On Figure \ref{varsolex} can be read the explicit value of the variational solution for small time $t$, which is given by the unique continuous section in the wavefront. The plain curve represents a shock of the variational solution, whereas the different expressions coincide $\cc^1$-continuously along the dotted curves. One can show that the variational solution denies the Hamilton--Jacobi equation in the viscosity sense along the thick portion of the shock, and also that it does satisfy the Hamilton--Jacobi equation in the viscosity sense everywhere except on this portion. For the purpose of this article, it is enough to show that the variational solution denies the Hamilton--Jacobi equation in the viscosity sense along the parabola circled in Figure \ref{varsolex}. This is included in the domain concerned by Proposition \ref{exdim2}, which can be proved by using an efficient convexity argument that spares ourselves many computations.

\begin{proof}[Proof of Proposition \ref{exdim2}] Using general arguments stated in Section \ref{proofvarmin}, we are first going to prove that
	\[R^t_0 u(q_1,q_2)= \min_{c\in[a,b]} u_c(t,q_1,q_2) \,\;\,\; \forall\, t\geq 0, (q_1,q_2) \in \rr^2,  \]
	where $u_c(t,q_1,q_2)= c(f(q_1+ct)-q_2)$ is the unique $\cc^2$ solution of the Cauchy problem associated with $H(p_1,p_2)=p_1p_2$ and the initial condition $u_c^0:(q_1,q_2)\mapsto c(f(q_1)-q_2)$. 
	
 We want to apply Proposition \ref{minf}, observing that $u= \min_{c\in [a,b]} u^0_c$. To do so, we only need to check that the family $\{u_c^0,c\in [a,b]\}$ satisfies the conditions of Lemma \ref{lem}, \emph{i.e.} that for all $(q,p)$ in the graph of the Clarke derivative $\dd u$, there exists $c\in [a,b]$ such that $u_c^0(q)=u(q)$ and $du_c^0(q)=p$.
	
	Let us compute the Clarke derivative of $u$. If $f(q_1)> q_2$, $u(q_1,q_2)= a(f(q_1)-q_2)$ on a neighbourhood of $(q_1,q_2)$, hence $\dd u(q_1,q_2)$ is reduced to the point $a \left(\begin{array}{c}
	f'(q_1)\\
	-1
	\end{array}\right)$ which is also the derivative of $u_a^0$ at $(q_1,q_2)$. 
	If $f(q_1)<q_2$, $\dd u(q_1,q_2)$ is reduced to the point  $b\left(\begin{array}{c}
	f'(q_1)\\
	-1
	\end{array}\right)$ which is also the derivative of $u_b^0$ at $(q_1,q_2)$. 
	If $f(q_1)=q_2$, $\dd u(q_1,q_2)$ is the segment $\left\{c\left(\begin{array}{c}
	f'(q_1)\\
	-1
	\end{array}\right),c\in [a,b] \right\}$. For all $c \in [a,b]$, $c\left(\begin{array}{c}
	f'(q_1)\\
	-1
	\end{array}\right)$ is the derivative of $u_c^0$ at the point $(q_1,q_2=f(q_1))$. 
	
	We hence proved that the family $\{u_c^0,c\in [a,b]\}$ satisfies the condition of Lemma \ref{lem}, hence by Proposition \ref{minf}
	\[R^t_0 u(q_1,q_2)= \min_{c\in[a,b]} u_c(t,q_1,q_2) \,\;\,\; \forall\, t\geq 0, (q_1,q_2) \in \rr^2.\]

Now, for all $-1 < q_1 < -bt$, $f(q_1+ct)=(q_1+ct)^2$ since $f(x)=x^2$ for $x$ in $[-1,1]$, and $c \in (0,b]$. Hence if $-1 < q_1 < -bt$,
	\[R^t_0 u(q_1,q_2)= \min_{c\in[a,b]} c\left( (q_1+ct)^2-q_2  \right).\]
The second derivative of $g:c \mapsto  c\left( (q_1+ct)^2-q_2  \right)$ is $g''(c) = 2t(2q_1+3ct)$. Hence if $q_1 < -\frac{3b}{2}t$, $g$ is concave on $[a,b]$ and the minimum defining $R^t_0 u(q_1,q_2)$ is attained at an endpoint of $[a,b]$.

Thus, we proved that for $-1<q_1< -\frac{3b}{2}t$,
\[R^t_0 u(q_1,q_2)= \min \left(a( (q_1+at)^2-q_2), b( (q_1+bt)^2-q_2)  \right).\]

\end{proof}	
	
\begin{proof}[Proof of Proposition \ref{contrex}]	
		Let $b>0$, $a \in (b/2,b)$ and $u$ be defined as in Proposition \ref{exdim2}: $f$ is a compactly supported $\cc^2$ function coinciding with $x \mapsto x^2$ on $[-1,1]$ and \[u(q_1,q_2)=\min\left(a(f(q_1)-q_2),b(f(q_1)-q_2)\right).\] 
		
		We define $u_a:(t,q_1,q_2) \mapsto a((q_1+at)^2-q_2)$ and $u_b:(t,q_1,q_2) \mapsto b((q_1+bt)^2-q_2)$ (note that the notations slightly differ from the previous proof), so that
		Proposition \ref{exdim2} gives that for $-1<q_1< -\frac{3b}{2}t$,
		\[R^t_0 u(q_1,q_2) = \min(u_a(t,q_1,q_2),u_b(t,q_1,q_2)).\]

	Let us prove that this variational solution denies the Hamilton--Jacobi equation at the point $\left(t,q_1,q_2\right)$ if
	 \[ \left\{\begin{array}{l}
q_2=q_1^2+2(a+b)tq_1 + t^2 (a^2+ab+b^2),\\
-1 <q_1<-\frac{3b}{2}t,\\
-(a+b)t< q_1.
	 \end{array}   \right. \]
	This corresponds to the piece of parabola circled on Figure \ref{varsolex}, which exists only if $a>b/2$ and $t< 2/3b$. Note that the first line is just an equation of this parabola, which is obtained by solving $u_a=u_b$.
	
Let us exhibit a test function denying the viscosity equation:  we define the mean function $\phi=\frac{1}{2}(u_a+u_b)$
	which is $\cc^1$, larger than $\min(u_a,u_b)$ on a neighbourhood of $(t,q_1,q_2)$ and equal to it at $(t,q_1,q_2)$ since $u_a(t,q_1,q_2)=u_b(t,q_1,q_2)$, so that $R^t_0 u-\phi$ attains a local maximum at $(t,q_1,q_2)$.
	The derivatives of $\phi$ are given by \[
	\begin{array}{l}
	\dd_t \phi(t,q_1,q_2)=a^2 (q_1+at)+b^2(q_1+bt),\\
	\dd_{q_1} \phi(t,q_1,q_2) = a (q_1+at) + b(q_1+bt),\\
	\dd_{q_2} \phi(t,q_1,q_2)= -\frac{1}{2}(a+b).
	\end{array}
	\]
	We compute \[\begin{split}
	\dd_t \phi(t,q_1,q_2) &+ H(\dd_q \phi(t,q_1,q_2))\\
	&= a^2 (q_1+at)+b^2(q_1+bt) -\frac{1}{2}(a+b)\left(a (q_1+at) + b(q_1+bt)\right)\\
	&=\frac{1}{2}(a-b)^2(at+bt+q_1)>0
	\end{split}
	\]
	when $q_1 > -(a+b)t$, and as a consequence the variational solution is not a viscosity subsolution at the point $(t,q_1,q_2)$.
	
Note that $b$ can be chosen as small as needed, and hence for all $L$ we are able to take the initial condition $u$ $L$-Lipschitz and $L$-semiconcave, with $b\leq L$. The previous work shows that for all $t< 2/3b$, the variational solution denies the Hamilton--Jacobi solution in the viscosity sense at some point $(t,q)$. But since $L\geq b$, $1/2L \leq 2/3b$ and we hence proved Proposition \ref{contrex}.
\end{proof}

\section{Proof of Theorem \ref{joukrec}} \label{u0C2}
In this section we will deduce Theorem \ref{joukrec} from Corollary \ref{joukrecsc}. To do so, we approach the Lipschitz initial condition of Corollary \ref{joukrecsc} by a smooth initial condition, keeping the Hausdorff distance between the (Clarke) derivatives small. We will use elementary properties of the Hausdorff distance, stated in Lemma \ref{contdh} and Lemma \ref{ete}, and proved for completeness.

The Hausdorff distance  $d_{\rm Haus}$  is defined (though not necessarily finite) by
\[ d_{\rm Haus}(X,Y)= \sup \left( \sup_{x\in X} d(x,Y)\,,\,\sup_{y\in Y} d(y,X)\right) \]
for $X$ and $Y$ closed subsets of a metric space $(E,d)$ ($d$ being the euclidean distance on $\rr^d$ in our context). 
The following approximation result is proved in \cite{rifford}, Theorem $2.2$ and its Corollary $2.1$:
\begin{theo}\label{rifczarn}
	If $u:\rr^d \to \rr$ is locally Lipschitz, there exists a sequence of smooth functions $u_n$ such that
	\[\begin{array}{l}
	\lim_{n \to \infty} \|u_n - u\|_{\infty} = 0,\\
	\lim_{n \to \infty} d_{\rm Haus} \left({\rm graph}(d u_n ), {\rm graph}(\dd u)\right) = 0,
	\end{array}\]
where $\dd$ denotes the Clarke derivative.
\end{theo}

Here is a sketch of the proof: for  $H$ an integrable non convex non concave Hamiltonian with bounded second derivative, Corollary \ref{joukrecsc} gives a Lipschitz initial condition $u_L$ such that the graph of the viscosity solution is not included in the wavefront $\ff_{u_L}$ for some time $t>0$. We are going to approach $u_L$ by a Lipschitz smooth function $u$ such that both the viscosity solutions at time $t$ are close, and the Hausdorff distance between the wavefronts at time $t$ is small. The following enhanced triangle inequality will conclude that the graph of the viscosity solution associated with $u$ is not included in the wavefront $\ff_u$.

	\begin{lem}[Enhanced triangle inequality] \label{ete} If $(E,d)$ is a metric space, $X$ and $Y$ are subsets of $E$, then for all $x$ and $y$ in $E$,
		\[d(x,X)\leq d(x,y) + d(y,Y) + d_{\rm Haus} (X,Y).\]
	\end{lem}
	\begin{proof}
		The triangle inequality for $d$ gives that for all $x$, $\tilde{x}$ and $y$, $d(x,\tilde{x}) \leq d(x,y) + d(y,\tilde{x})$, and taking the infimum for $\tilde{x}$ on $X$ gives
		\begin{equation}\label{ineq1}
		d(x,X) \leq d(x,y)+ d(y,X)   \,\; \forall x,y \in E.    
		\end{equation}
		We change the variables in \eqref{ineq1}: for all $y$ and $\tilde{y}$, 
		\[d(y,X) \leq d(y,\tilde{y})+ d(\tilde{y},X).\]
		If $\tilde{y}$ is in $Y$, by definition of the Hausdorff distance we get 
		\[d(y,X) \leq d(y,\tilde{y}) + d_{\rm Haus }(X,Y) \]
		and taking the infimum for $\tilde{y}$ on $Y$ gives that \[
		d(y,X) \leq d(y,Y) +  d_{\rm Haus }(X,Y).\]
		We conclude by injecting this last inequality into \eqref{ineq1}.
	\end{proof}

To bound the Hausdorff distance between the wavefronts, we will describe the wavefront at time $t$ as the image of the (Clarke) derivative of the initial condition by a suitable function $\psi$ depending on the initial condition, which will allow to apply the following elementary continuity result for the Hausdorff distance.  

\begin{lem}[Continuity for the Hausdorff distance] \label{contdh} Let $f,g : (F,\tilde{d})\mapsto (E,d)$ be two functions between two topological spaces, and $X$ and $Y$ be two subsets of $F$. Then
	\begin{enumerate}
		\item if $d\left(f(x),g(x)\right) \leq a$ for all $x$ in $X$, then $d_{\rm Haus} \left(f(X),g(X)\right) \leq a$,
		\item if $f$ is uniformly continuous on $X$, \emph{i.e.} for all $\alpha>0$, there exists $\ep>0$ such that for all $(x,y) \in X$, $\tilde{d}(x,y)< \ep$ implies $d(f(x),f(y)) < \alpha$, then 
		\[\tilde{d}_{\rm Haus} (X,Y) < \ep \implies d_{\rm Haus} \left(f(X),f(Y)\right)< \alpha. \]
	\end{enumerate}
\end{lem}
\begin{proof}[Proof of Lemma \ref{contdh}] \begin{enumerate}
		\item By definition of the Hausdorff distance, it is enough to observe that $d\left(f(x),g(X)\right) \leq a$ for all $x$ in $X$, since this quantity is smaller than $d(f(x),g(x))$.
		\item Using the symmetry of the definition of $d_{\rm Haus}$, it is enough to prove that if $\tilde{d}_{\rm Haus} (X,Y) <\ep$, $d(f(x),f(Y))<\alpha$ for all $x$ in $X$. For all $x$ in $X$, there exists a sequence $y_n$ in $Y$ such that $\tilde{d}(x,y_n) \underset{n\to \infty}{\to}\tilde{d}(x,Y)$. Since $\tilde{d}(x,Y)\leq \tilde{d}_{\rm Haus} (X,Y)$, this implies that $\tilde{d}(x,y_n)< \ep$ for $n$ large enough, and the uniform continuity of $f$ gives that  $d(f(x),f(y_n))< \alpha$ for $n$ large enough, hence $d\left(f(x),f(Y)\right) < \alpha$.
	\end{enumerate}
\end{proof}

\begin{proof}[Proof of Theorem \ref{joukrec}] Let $H$ be an integrable non convex non concave Hamiltonian with bounded second derivative. Corollary \ref{joukrecsc} gives a Lipschitz initial condition $u_L$ for which there exist $t>0$ and $q$ such that
	\[d\left((q,V^t_0 u_L(q)), \ff^t_{u_L} \right) >0\]
where $\ff^t_{u_L}$ denotes the section of $\ff_{u_L}$ at time $t$. We denote by $\alpha$ this positive quantity.
	
	Let us denote by $L$ the Lipschitz constant of $u_L$. 
	
	We propose an other description of the wavefront at time $t$: for all Lipschitz function $v$, we define \[
	\begin{array}{cccc}
	\psi^t_v : & T^\star \rr^d & \to & \rr^d \times \rr\\
	&	(q,p) &\mapsto &\left(q + t\nabla H (p), v(q)+t (p\cdot \nabla H (p) -H(p))\right), 
	\end{array}\]
	in such a way that $\ff_v^t=\psi^t_v({\rm graph}(\dd v))$ (see \eqref{defwf2} for a comparison).
	
	Note that $\psi^t_{v}$ is Lipschitz, hence uniformly continuous on every $\rr^d \times \{\|p\| \leq R\}$ for $R>0$: it is Lipschitz with respect to $q$ because $v$ is, and its derivative with respect to $p$, $\left(td^2H(p), t p \cdot d^2H(p)\right)$, is bounded on this set since $d^2H$ is bounded.
	
	The uniform continuity of $\psi^t_{u_L}$ on $\rr^d \times \{p\leq L+1\}$ gives a $\ep \in (0,1)$ such that 
	\[ \left\{ \begin{array}{l}
	\|(q,p)-(\tilde{q},\tilde{p})\|< \ep, \\
	\|p\|,\|\tilde{p}\|\leq L+1,
	\end{array}\right. \implies \|\psi^t_{u_L}(q,p)-\psi^t_{u_L}(\tilde{q},\tilde{p})\| < \alpha/4.\]
	
	By Theorem \ref{rifczarn}, there exists a smooth function $u$ such that \begin{align}
	\label{5}		\|u- u_L\|_\infty & < \alpha/4, \\
	\label{6}	 d_{\rm Haus} \left({\rm graph}(d u ), {\rm graph}(\dd u_L)\right) & < \ep .
	\end{align}
	Note that since $\ep \in (0,1)$, $u$ is $(L+1)$-Lipschitz.
	
		On the one hand, Proposition \ref{visc} gives the comparison between the viscosity solutions:
		\[\|V^t_0 u-V^t_0 u_L\|_\infty \leq \|u-u_L\|_\infty  \leq \alpha/4.\]
			
	On the other hand, we estimate the Hausdorff distance between the wavefronts, using the definition of $\psi$:
	\[
\begin{split}
d_{\rm Haus} \left(\ff^t_u,\ff^t_{u_L}\right) =\,	& d_{\rm Haus} \left(\psi^t_{u}({\rm graph}(du)),\psi^t_{u_L}({\rm graph}(\dd u_L))\right)\\
&\leq d_{\rm Haus}  \left(\psi^t_{u}({\rm graph}(du)),\psi^t_{u_L}({\rm graph}(du))\right) \\
	& \,\;\,\;	+ 	d_{\rm Haus} \left(\psi^t_{u_L}({\rm graph}(du)),\psi^t_{u_L}({\rm graph}(\dd u_L))\right). 
\end{split}
	\]
	The first part of Lemma \ref{contdh} applied with $f=\psi^t_{u_L}$, $g=\psi^t_u$, $X=  {\rm graph}(d u)$ gives that the first term of the right hand side is bounded by $\|\psi^t_u- \psi^t_{u_L}\|_\infty =\|u-u_L\|_\infty \leq \alpha/4$. 
	
	The second part of Lemma \ref{contdh} applied with $f=\psi^t_{u_L}$,  $X= {\rm graph}(\dd u_L)$ and $Y= {\rm graph}(d u)$  gives that the second term of the right hand side is smaller than $\alpha/4$, by uniform continuity of $\psi^t_{u_L}$, since ${\rm graph}(du)$ and ${\rm graph}(\dd u_L)$ are  both contained in $\rr^d \times \{p\leq L+1\}$ and are $\ep$-close for the Hausdorff distance, see \eqref{6}. We hence proved that
	\[d_{\rm Haus} \left(\ff^t_u,\ff^t_{u_L}\right) \leq \alpha/2.\]

	Let us now apply Lemma \ref{ete} with $x=(q,V^t_0 u_L(q))$, $y=(q,V^t_0 u_L(q))$, $X=\ff^t_{u_L}$ and $Y=\ff^t_u$:
	\[ \begin{split}
	\alpha = & \,\;d\left((q,V^t_0 u_L(q)), \ff^t_{u_L} \right) \\
&	\leq \underbrace{d\left((q,V^t_0 u_L(q)),(q,V^t_0 u(q))\right)}_{\leq \|V^t_0 u_L-V^t_0 u\|_\infty\leq \alpha/4} + d\left((q,V^t_0 u(q)),\ff^t_u\right) + \underbrace{d_{\rm Haus} \left(\ff^t_{u_L},\ff^t_u\right)}_{\leq \alpha/2}.
	\end{split}
   \]
	As a consequence, $d\left((q,V^t_0 u(q)),\ff^t_u\right)\geq \alpha/4>0$ and the graph of the viscosity solution associated with the smooth initial condition $u$ is not contained in the wavefront $\ff_u$.	
\end{proof}

\section{Semiconcavity arguments} \label{proofvarmin}

This section contains the proofs of Propositions \ref{varmin} and \ref{v<v}, as well as an additional Proposition \ref{minf} used in the proof of the two-dimensional case (see \S \ref{ce2}). The three proofs rely on the following lemma, proved in \cite{bernard2} (Lemma $6$):	
	\begin{lem}\label{lem}
		If $u$ is a Lipschitz and $B$-semiconcave function on $\rr^d$, there exists a family $F$ of $\cc^2$ equi-Lipschitz functions with second derivatives bounded by $B$	such that:\begin{itemize}
			\item $u(q)=\min_{f\in F}f(q)$ for any $q$,
			\item for each $q$ in $\rr^d$ and $p$ in $\dd u(q)$, there exists $f$ in $F$ such that $\left\{\begin{array}{c}
			f(q)=u(q),\\
			df(q)=p.
			\end{array}\right.$ 
		\end{itemize}
	\end{lem}

\begin{proof}[Proof of Proposition \ref{varmin}]
 Proposition \ref{extvar} states that the variational solution gives a section of the generalized wavefront. As a consequence \[
	R^t_0 u_0(q)\geq \inf\left\{ u_0(q_0)+ \aaa^t_0(\gamma)\left|\begin{array}{c}
	(q_0,p_0)\in \rr^d\times  \rr^d, \\
	p_0 \in \dd u_0(q_0),\\
	Q^t_0(q_0,p_0)=q.
	\end{array} \right\}\right. .\] 
	If $u_0$ is $L$-Lipschitz and $B$-semiconcave, take $T$ such that the method of characteristics is valid ($T=1/BC$ if $H$ is integrable).
	Let us fix definitively $q$, $q_0$, $p_0 \in \dd u_0(q_0)$ and $0\leq t \leq T$ such that $Q^t_0(q_0,p_0)=q$ and show that $R^t_0 u_0(q) \leq u_0(q_0)+\aaa^t_0(\gamma)$ where $\gamma$ is the Hamiltonian trajectory issued from $(q_0,p_0)$.
	
	Lemma \ref{lem} gives a $\cc^2$ function $f_0$ of $F$ such that $f_0(q_0)=u_0(q_0)$ and $df_0(q_0)=p_0$. 
	Since this function is $\cc^2$ with second derivative bounded by $B$, the method of characteristics gives that $q_0$ is the only point such that $Q^t_0(q_0,df_0(q_0))=q$, and the variational operator applied to the initial condition $f_0$ gives necessarily the $\cc^2$ solution:
	\[R^t_0 f_0 (t,q)=f_0(q_0)+\aaa^t_0(\gamma).\]
	
	But by definition of $F$, $f_0$ is larger than $u_0$ on $\rr^d$, and the monotonicity of the variational operator brings the conclusion: \[
	R^t_0 u_0(q) \leq R^t_0 f_0(q)=f_0(q_0)+\aaa^t_0(\gamma)=u_0(q_0)+\aaa^t_0(\gamma).
	\]

\end{proof}
\begin{proof}[Proof of Proposition \ref{v<v}.]
	Take $T$ such that the method of characteristics is valid (for example $T=1/BC$ if $H$ is integrable).
	
	If $t$ and $q$ are fixed, Proposition \ref{extvar} gives the existence of $(q_0,p_0)$ in $gr(\dd u_0)$ such that $Q^t_0(q_0,p_0)=q$ and that $R^t_0 u_0(q) = u_0(q_0)+\aaa^t_0(\gamma)$ where $\gamma$ is the Hamiltonian trajectory issued from $(q_0,p_0)$.
	
	Lemma \ref{lem} gives a $\cc^2$ function $f_0$ of $F$ such that $f_0(q_0)=u_0(q_0)$ and $df_0(q_0)=p_0$. 
	The method of characteristics states that there exists on $[0,T]\times \rr^d$ a unique $\cc^2$ solution of the \eqref{HJ} equation with initial condition $f_0$, which satisfies in particular \[f(t,q)=f_0(q_0)+ \aaa^t_0(\gamma).\] 	
	Since a $\cc^1$ solution is a viscosity solution, the uniqueness of viscosity solutions hence gives that $V^t_0 f = f(t,\cdot)$ for all $t$ in $(0,T)$, and in particular
	\[V^t_0 f_0 (q)=f(t,q)=f_0(q_0)+\aaa^t_0(\gamma).\]
	
	But by definition of $F$, $f_0$ is larger than $u_0$ on $\rr^d$, and the monotonicity of the viscosity operator $V^t_0$ brings the conclusion:
	\[V^t_0 u_0(q) \leq V^t_0 f_0 (t,q) = f_0(q_0)+\aaa^t_0(\gamma)=R^t_0 u_0(q).\]
	
	Since $(t,q)\mapsto R^t_0 u_0(q)$ is less or equal than any variational solution as long as $t<T$ (Proposition \ref{varmin}), this implies that for all variational solution $g$, $V^t_0 u_0(q)\leq g(t,q)$ on $[0,T]\times \rr^d$.
\end{proof}

We end this paragraph with another result of the same flavor, used in the proof of Proposition \ref{contrex}.

\begin{prop}\label{minf} Let $F$ be as in Lemma \ref{lem} and $u= \min_{f\in F} f$. If $T>0$ denotes a time of shared existence of $\cc^2$ solutions for initial conditions in $F$, and $u_f$ denotes the $\cc^2$ solution of the Hamilton--Jacobi equation associated with the $\cc^2$ initial condition $f$, then for all $0\leq t\leq T$
	\[R^t_0 u(q) = \min_{f\in F} u_f(t,q).\]
\end{prop}
\begin{proof} Since $u \leq f$ for all $f$ in $F$, the monotonicity of the variational operator guarantees that $R^t_0 u(q) \leq  \min_{f\in F} R^t_0f(q)$. The method of characteristics implies that the variational operator is given by the classical solution if it exists, hence $R^t_0 f(q)=u_f(t,q)$ for all $t$ in $[0,T]$ and thus
	\begin{equation}\label{ineq}R^t_0 u(q) \leq \min_{f\in F} u_f(t,q).\end{equation}
Now, for all $(t,q)$, the variational property gives the existence of a $(q_0,p_0)$ in the graph of $\dd u$ such that 
\[R^t_0 u(q) = u(q_0) + \aaa^t_0(\gamma)\]
where $\gamma$ denotes the Hamiltonian trajectory issued from $(q_0,p_0)$. Since $F$ is as in Lemma \ref{lem}, there exists $f$ in $F$ such that $f(q_0)=u(q_0)$ and $df(q_0)=p_0$. The method of characteristics implies furthermore that $u_f(t,q)=f(q_0)+ \aaa^t_0(\gamma)$. Summing all this up, we get
\[R^t_0 u(q) = u(q_0) + \aaa^t_0(\gamma)=f(q_0) + \aaa^t_0(\gamma)= u_f(t,q)\]
and the inequality \eqref{ineq} is an equality.
\end{proof}

\subsubsection*{Acknowledgement} The author is very grateful to P. Bernard and J.-C. Sikorav who helped organizing this paper and proposed many improvements in the argumentation, as well as to M. Zavidovique for a fruitful discussion. The research leading to these results has received funding from the European Research Council under the European Union’s Seventh Framework Programme (FP/2007-2013) / ERC Grant Agreement 307062 and from the French National Research Agency via ANR-12-BLAN-WKBHJ.

\bibliographystyle{alpha}
\bibliography{biblio} 

\begin{thebibliography}{Roo17b}

\bibitem[Ber13]{bernard2}
P.~Bernard.
\newblock Semi-concave singularities and the {H}amilton-{J}acobi equation.
\newblock {\em Regul. Chaotic Dyn.}, 18(6):674--685, 2013.

\bibitem[CC18]{canchesurvey}
P.~{Cannarsa} and W.~{Cheng}.
\newblock {On and beyond propagation of singularities of viscosity solutions}.
\newblock {\em ArXiv e-prints}, May 2018.

\bibitem[CCF17]{canchefat}
Piermarco Cannarsa, Wei Cheng, and Albert Fathi.
\newblock On the topology of the set of singularities of a solution to the
  {H}amilton-{J}acobi equation.
\newblock {\em C. R. Math. Acad. Sci. Paris}, 355(2):176--180, 2017.

\bibitem[CEL84]{CEL}
M.~G. Crandall, L.~C. Evans, and P.-L. Lions.
\newblock Some properties of viscosity solutions of {H}amilton-{J}acobi
  equations.
\newblock {\em Trans. Amer. Math. Soc.}, 282(2):487--502, 1984.

\bibitem[Che75]{chenciner}
A.~Chenciner.
\newblock Aspects géométriques de l’études des chocs dans les lois de
  conservation.
\newblock {\em Problèmes d’évolution non linéaires, Séminaire de Nice},
  (15):1--37, 1975.

\bibitem[CIL92]{guide}
M.~G. Crandall, H.~Ishii, and P.-L. Lions.
\newblock User's guide to viscosity solutions of second order partial
  differential equations.
\newblock {\em Bull. Amer. Math. Soc. (N.S.)}, 27(1):1--67, 1992.

\bibitem[CL83]{cr&lions83}
M.~G. Crandall and P.-L. Lions.
\newblock Viscosity solutions of {H}amilton-{J}acobi equations.
\newblock {\em Trans. Amer. Math. Soc.}, 277(1):1--42, 1983.

\bibitem[CMS15]{canmazsin}
Piermarco Cannarsa, Marco Mazzola, and Carlo Sinestrari.
\newblock Global propagation of singularities for time dependent
  {H}amilton-{J}acobi equations.
\newblock {\em Discrete Contin. Dyn. Syst.}, 35(9):4225--4239, 2015.

\bibitem[CR06]{rifford}
M.-O. Czarnecki and L.~Rifford.
\newblock Approximation and regularization of {L}ipschitz functions:
  convergence of the gradients.
\newblock {\em Trans. Amer. Math. Soc.}, 358(10):4467--4520 (electronic), 2006.

\bibitem[CV08]{carvit}
F.~Cardin and C.~Viterbo.
\newblock Commuting {H}amiltonians and {H}amilton-{J}acobi multi-time
  equations.
\newblock {\em Duke Math. J.}, 144(2):235--284, 2008.

\bibitem[DZ15]{davzav}
A.~Davini and M.~Zavidovique.
\newblock On the (non) existence of viscosity solutions of multi-time
  {H}amilton-{J}acobi equations.
\newblock {\em J. Differential Equations}, 258(2):362--378, 2015.

\bibitem[IK96]{izukos}
S.~Izumiya and G.~T. Kossioris.
\newblock Formation of singularities for viscosity solutions of
  {H}amilton-{J}acobi equations.
\newblock In {\em Singularities and differential equations ({W}arsaw, 1993)},
  volume~33 of {\em Banach Center Publ.}, pages 127--148. Polish Acad. Sci.
  Inst. Math., Warsaw, 1996.

\bibitem[Jou91]{jou}
T.~Joukovskaia.
\newblock {\em Singularités de Minimax et Solutions Faibles d’Équations aux
  Dérivées Partielles}.
\newblock 1991.
\newblock Thèse de Doctorat, Université de Paris VII, Denis Diderot.

\bibitem[Kos93]{kos}
G.~T. Kossioris.
\newblock Formation of singularities for viscosity solutions of
  {H}amilton-{J}acobi equations in one space variable.
\newblock {\em Comm. Partial Differential Equations}, 18(5-6):747--770, 1993.

\bibitem[Ole59]{oleinik}
O.~A. Ole{\u\i}nik.
\newblock Uniqueness and stability of the generalized solution of the {C}auchy
  problem for a quasi-linear equation.
\newblock {\em Uspekhi Mat. Nauk}, 14(2):165--170, 1959.
\newblock russian only.

\bibitem[Roo17a]{roos1}
Valentine Roos.
\newblock {Variational and viscosity operators for the evolutionary
  Hamilton–Jacobi equation}.
\newblock {\em {Communications in Contemporary Mathematics}}, 2017.

\bibitem[Roo17b]{mathz}
Valentine Roos.
\newblock {\em {Variational and viscosity solutions of the Hamilton-Jacobi
  equation}}.
\newblock Theses, {PSL Research University}, June 2017.

\bibitem[Vit96]{viterboX}
C.~Viterbo.
\newblock Solutions of {H}amilton-{J}acobi equations and symplectic geometry.
  {A}ddendum to: {\it {S}\'eminaire sur les \'Equations aux {D}\'eriv\'ees
  {P}artielles. 1994--1995} [\'ecole {P}olytech., {P}alaiseau, 1995;
  {MR}1362548 (96g:35001)].
\newblock In {\em S\'eminaire sur les \'Equations aux {D}\'eriv\'ees
  {P}artielles, 1995--1996}, S\'emin. \'Equ. D\'eriv. Partielles, page~8.
  \'Ecole Polytech., Palaiseau, 1996.

\bibitem[Wei14]{wei}
Q.~Wei.
\newblock Viscosity solution of the {H}amilton-{J}acobi equation by a limiting
  minimax method.
\newblock {\em Nonlinearity}, 27(1):17--41, 2014.

\end{thebibliography}

\end{document}